%% file: main.tex
\documentclass{article}

\usepackage{tikz-cd}

\usepackage{graphics,amsmath,amsfonts,amssymb,graphicx,url,color,amsthm,cite}
\usepackage{graphbox}
\usepackage[font=small]{caption}
\usepackage
{hyperref}

\usepackage{import,pdfpages}

\newcommand{\s}{\small}

\usepackage{calc}
\newlength{\notewidth}
\newtheorem{lemma}{Lemma}
\newtheorem{proposition}[lemma]{Proposition}

\newtheorem{theorem}{Theorem}

\newtheorem{conjecture}{Conjecture}

\theoremstyle{definition}

\newtheorem{remark}[lemma]{Remark}

\newcommand{\R}{{\mathbb R}}
\newcommand{\Z}{{\mathbb Z}}
\newcommand{\RP}{{\mathbb {RP}}}

\newcommand{\be}{\begin{equation}}
\newcommand{\ee}{\end{equation}}

\renewcommand{\L}{\mathcal{L}}

\newcommand{\mn}{\medskip\noindent}

\begin{document}

\title{On cusps of caustics by reflection: a billiard variation on Jacobi's Last Geometric Statement}

\author{Gil Bor\footnote{
CIMAT, A.P. 402, Guanajuato, Gto. 36000, Mexico; 
{\em gil@cimat.mx}
}
\and
Serge Tabachnikov\footnote{
Department of Mathematics,
Penn State University, 
University Park, PA 16802;
{\em tabachni@math.psu.edu}}
}

\date{\today}
\maketitle


\begin{abstract} A point source of light is placed inside an oval. The  $n$-th caustic by reflection is the envelope of the light rays emanating from the light source after  $n$ reflections off the curve. We show that each of these caustics, for a generic point light source, has  at least 4 cusps. This  is a billiard variation on  Jacobi's Last Geometric Statement, concerning the number of cusps of the conjugate locus of a point on a convex surface. We present various proofs, using different ideas, including the  curve shortening flow and Legendrian knot theory.
\end{abstract}

\section{Introduction}
\subsection{Motivation and background}
The {\em conjugate locus} of a point on a surface is the locus of first conjugate points along  geodesics emanating from that point. 
In his ``Lectures on Dynamics" \cite{Cl}, published posthumously,
Jacobi stated that  the conjugate locus of a generic point on an ellipsoid has exactly four cusps.
This {\it Last Geometric Statement of Jacobi} was proved only in this century, see \cite{IK}. Indeed, as
recently as the end of the 20th century, Marcel Berger wrote \cite{Ber}:
\begin{quote}
{\small
... this latter assumption depends on the scandalously unproved Jacobi ``statement": the conjugate locus of a non-umbilical point of an ellipsoid has exactly four cusps.}
\end{quote}

A related result is that the conjugate locus of a generic point on a convex surface has at least four cusps, see \cite{Wa} for a recent proof. This theorem was attributed to C. Carath\'eodory (1912) by W. Blaschke (sect. 103 of \cite{Bl}), who presented a sketch of the proof. This theorem belongs to a long list of results that stem from and are motivated by the celebrated 4-vertex theorem of S. Mukhopadhyaya. See \cite{Ar,GTT} for surveys. 

The conjugate locus can be equivalently described as the locus of the first intersections of infinitesimally close geodesics emanating from a point. These geodesics may intersect more than once, and the loci of their intersections are known as second, third, etc., caustics of the point. Statements similar to Jacobi's statement and generalization to arbitrary convex surfaces about the  higher order caustics are still open. There is some experimental evidence that if the surface is an ellipsoid then, for a non-umbilic point, each such caustic has exactly four cusps, see \cite{Si} and Figure \ref{ellipsoids} from this paper (presented with permission). 

\begin{figure}[ht]
\centering
\includegraphics[width=.3\textwidth]{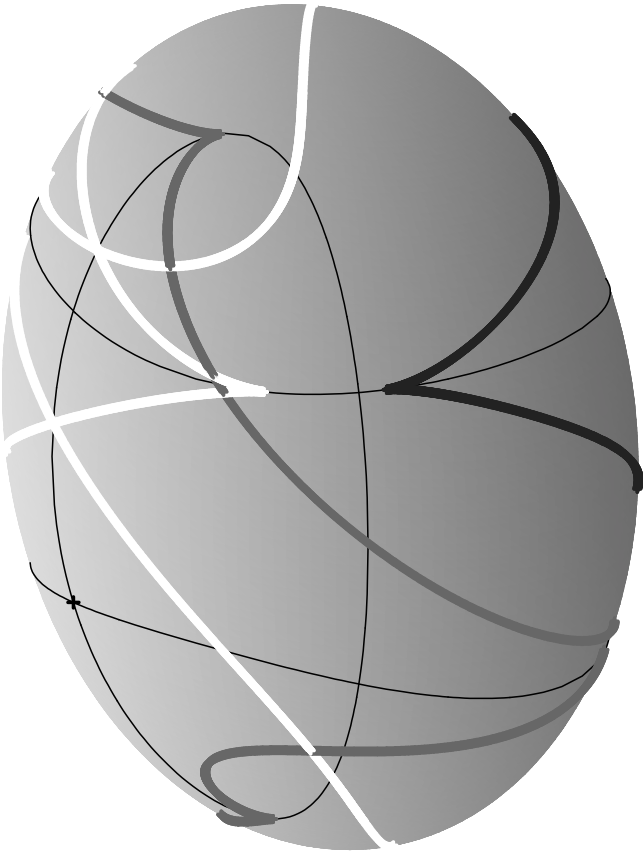}\hspace{.2\textwidth}
\includegraphics[width=.3\textwidth]{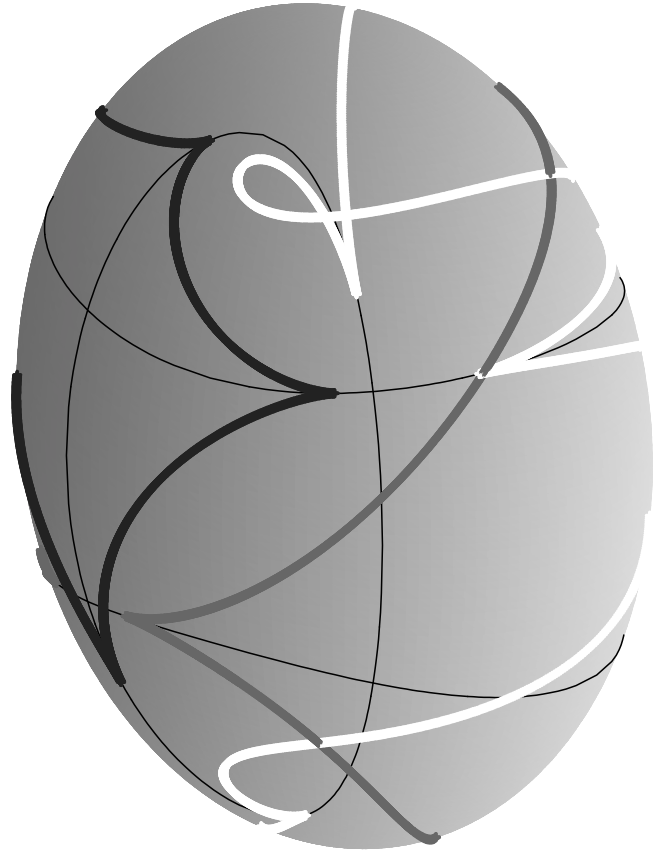}
\caption{The first three caustics of an non-umbilic point on an ellipsoid.}
\label{ellipsoids}
\end{figure}

In this article we consider a billiard version of this  problem. Let $\gamma$ be an oval (a smooth strictly convex closed curve in $\R^2$), the boundary of a billiard table or, equivalently, an ideal mirror. Let $O$ be a point inside $\gamma$, a source of light. For $n=1,2.\ldots$, the 1-parameter family of rays that have undergone $n$ optical reflections in $\gamma$ envelopes a curve $\Gamma_n$, the {\em $n$-th caustic by reflection}. See Figure \ref{fig:caustics}.

\begin{figure}[ht]
\centering
\def\svgwidth{.4\textwidth}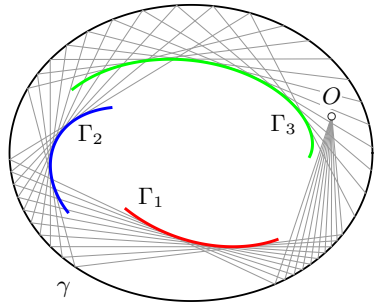
\caption{The   $n$-th caustic by reflection $\Gamma_n$ is the envelope of the family of rays emanating from $O$ that have undergone $n$ reflections by $\gamma$.}
\label{fig:caustics}
\end{figure}

These caustics may have singularities, generically, semi-cubical cusps. We always assume that the caustics $\Gamma_n$ are in general position in this sense. The singularities of caustics were thoroughly studied  by Bruce, Giblin, and Gibson; see \cite{BGG} and the references therein.

Figure \ref{fig:cc}a shows that a caustic by reflection may extend beyond the interior of $\gamma$, and furthermore,  it can be disconnected in the Euclidean plane; however, as the envelope of a 1-parameter family of lines, it is a connected curve in the projective plane $\RP^2$ (possibly, with singularities). Indeed, a 1-parameter family of lines is a curve in the space of lines, and the respective envelope is projectively dual to this curve.

\begin{figure}[ht]
\centering
\def\svgwidth{.45\textwidth}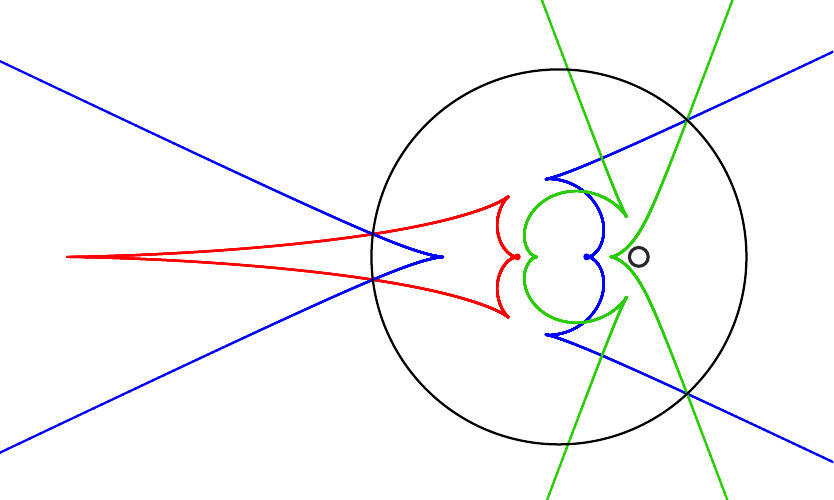
\hspace{.15\textwidth}
\def\svgwidth{.32\textwidth}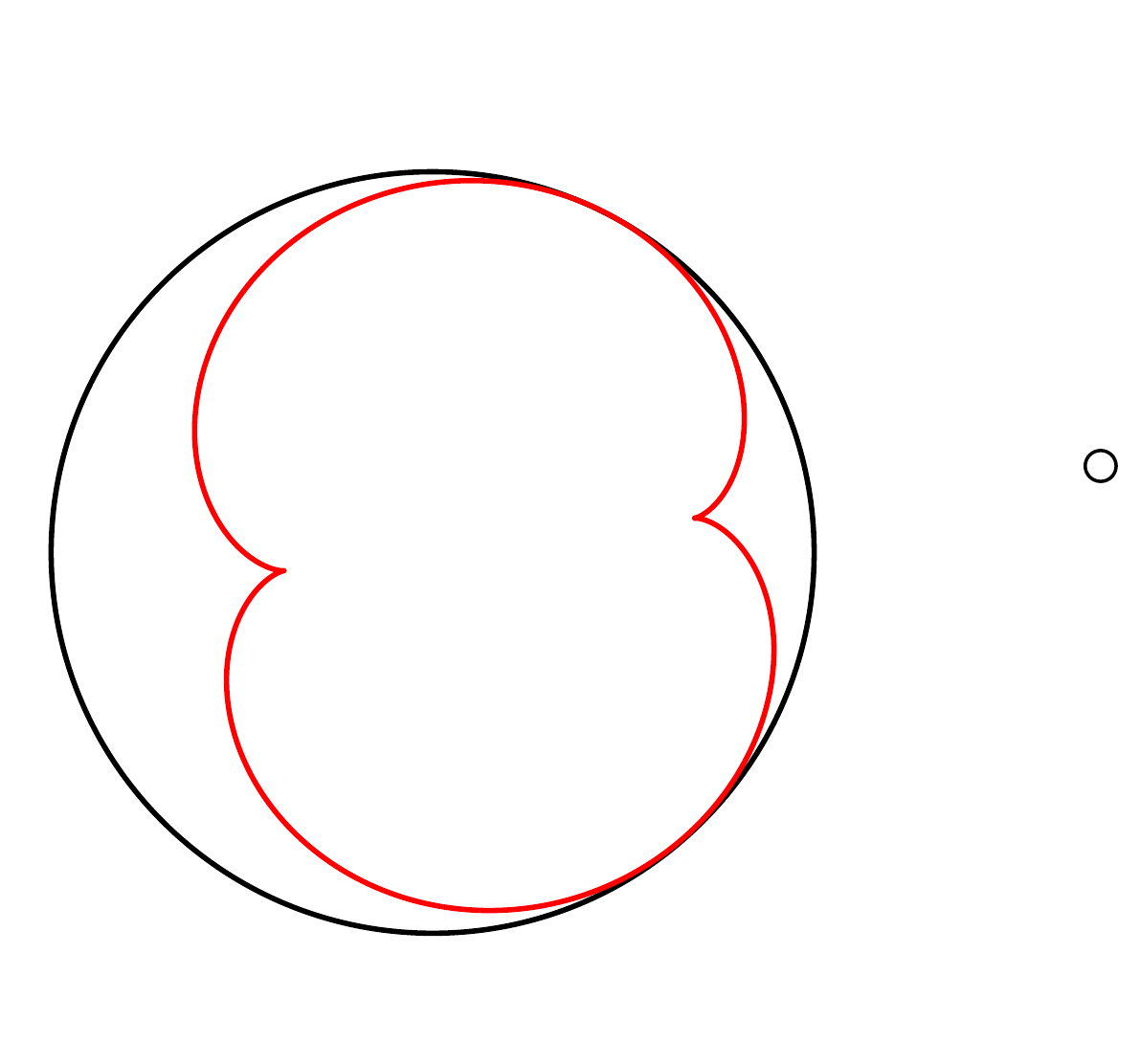

\caption{(a)  The first three caustics by reflection in a circle, showing 4 cusps on each of them (red=1st caustic, blue=2nd, green=3rd). The small circle is the light  source. (b) The first caustic by reflection  in a circle with an external source of light (each ray optically reflects at both intersection points with the circle). The caustic in this case has only 2 cusps. In this paper we  consider only an internal light source.}
\label{fig:cc}
\end{figure}

\subsection{The main result and two conjectures}
Our main result is as follows.

\begin{theorem} \label{thm:main}
For every oval $\gamma\subset\R^2$, a generic light source  inside $\gamma$ and  $n\ge 1$, the $n$-th caustic by reflection   $\Gamma_n\subset\RP^2$ has at least four cusps. 
\end{theorem}

We present  three proof sketches.

Let $\L$ be the space of directed lines in  $\R^2$. To each caustic  $\Gamma_n$ is associated its dual  curve  $C_n\subset\L$, corresponding to  the tangent lines along  $\Gamma_n$ (the rays of the $n$-th reflected beam). One can identify  $\L$ with the complement of the `north' and `south' pole of the unit sphere $S^2\subset\R^3$,  so that cusps of $\Gamma_n$ correspond to   inflection points of $C_n$ (points with vanishing   spherical geodesic curvature). Using  standard properties of convex billiards, we show that $C_n$ is a closed simple smooth curve in  $S^2$, intersecting every great circle. A theorem of B. Segre from 1968  \cite{Se,Wei}  states  that such a curve has at least four spherical inflection points, thus completing the proof of Theorem \ref{thm:main}. 

Another approach, starts with a realization of $\L$ as the vertical cylinder circumscribing $S^2$ and the curve $C_n \subset \L$ representing  the tangent lines of $\Gamma_n$.
Following S. Angenent \cite{An}, apply  
the curve shortening flow with respect to the flat metric on the cylinder to the curve $C_n$ to deform it  to the  graph of a function  $F:S^1\to\R$ with zero mean value. Spherical inflection points of $C_n$ correspond  to  the zeroes of $F''+F$, a function with vanishing constant  and first order Fourier terms. By the Sturm-Hurwitz theorem, it has at least four zeros.

Yet another approach is to use the relation between the cusps of the caustic $\Gamma_n$ and the vertices (critical points of the curvature) of its normal front $\Delta_n$, a closed planar curve whose normal lines, parametrized by $C_n$, are the lines tangent to $\Gamma_n$, see Figure \ref{fig:3curves}. The relation bewteen $\Gamma_n$ and $\Delta_n$ is the familiar relation between evolutes and involutes, see, e.g., \cite{GTT}.  

 \begin{figure}[ht]
\centering
\def\svgwidth{\textwidth}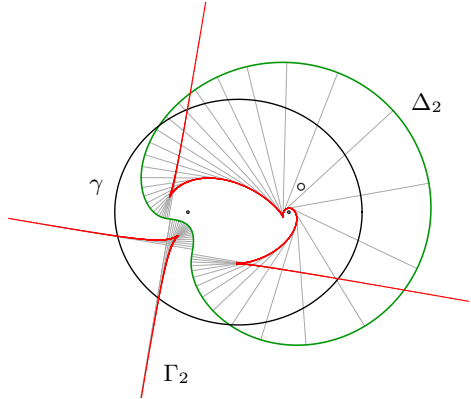
\caption{The 2nd caustics $\Gamma_2$ (red) with an involute $\Delta_2$ (green) for an elliptical billiard table $\gamma$ (black). The rays (gray) correspond to $C_2$, are normal to the wave front $\Delta_2$ and are tangent to $\Gamma_2$.  }
\label{fig:3curves}
\end{figure}

We show that $\Delta_n$ exists as a closed curve, possibly with cusps (in fact, there is a 1-parameter equidistant family of such curves).  To show that $\Delta_n$ has at least four  vertices, we use a theorem of  Chekhanov and Pushkar \cite{ChP}, stating that a planar cooriented closed wave front has at least four vertices, provided its Legendrian lift to the space of cooriented contact element in the plane  is a `Legendrian unknot', that is, is Legendrian isotopic to the Legendrian lift of a circle.

The rest of the article provides background information and details of these arguments. In the last section we present some generalizations of Theorem \ref{thm:main} to spherical and hyperbolic geometry, as well as ``projective billiards.''

\mn 

We present two conjectures; the first one is supported by experimental evidence, the  second one might be over-optimistic.

\begin{conjecture} \label{conj:one}
If $\gamma$ is an ellipse, then the caustic by reflection $\Gamma_n$ for a  light source inside $\gamma$ and different from a focus has exactly four cusps for every $n\ge 1$ (see Figure \ref{ellipse1}).
\end{conjecture}

This conjecture is only known to hold in the case of $n=1$ (the `catacaustic', see next section). 

\begin{figure}[ht]
\centering
\includegraphics[width=.3\textwidth]{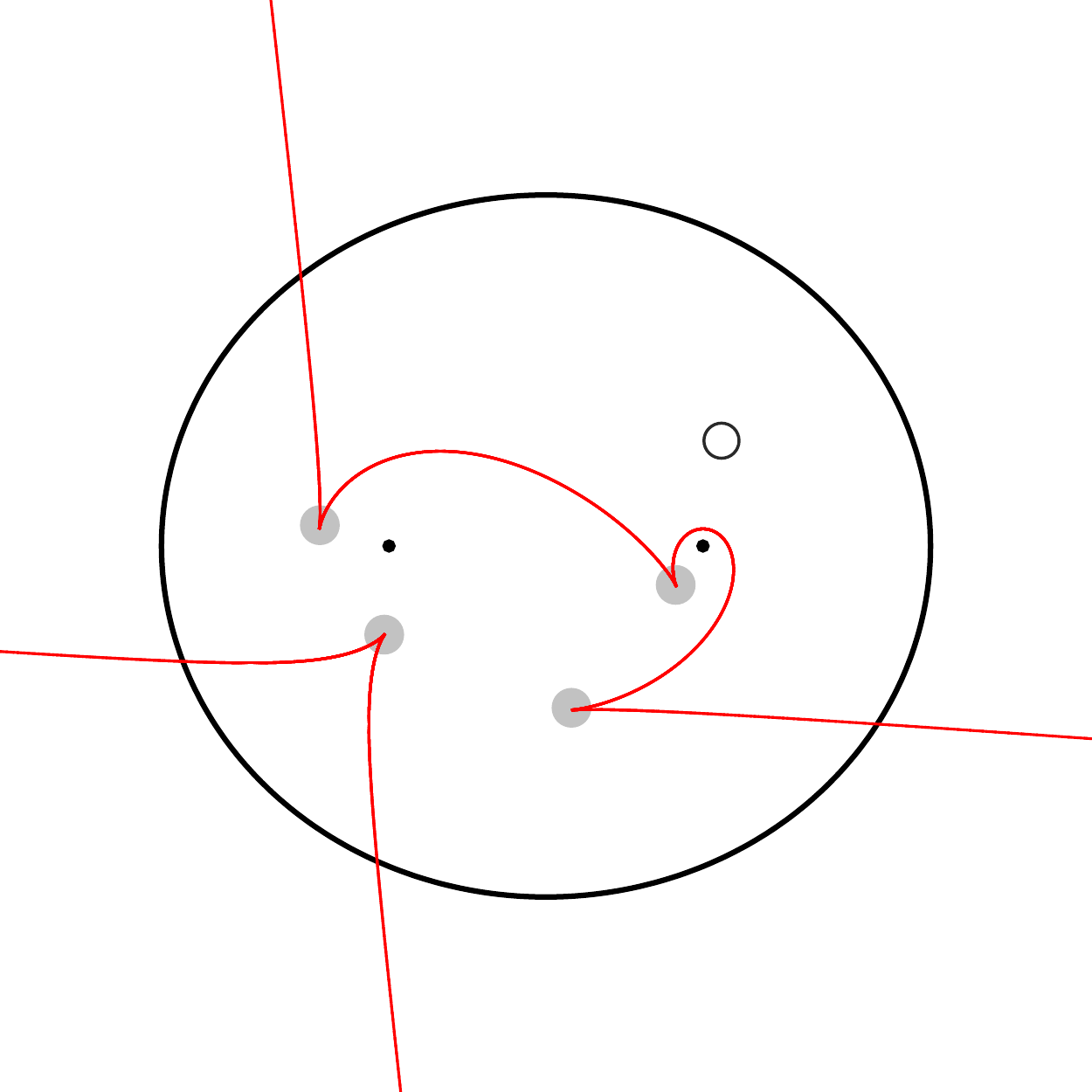}\quad
\includegraphics[width=.3\textwidth]{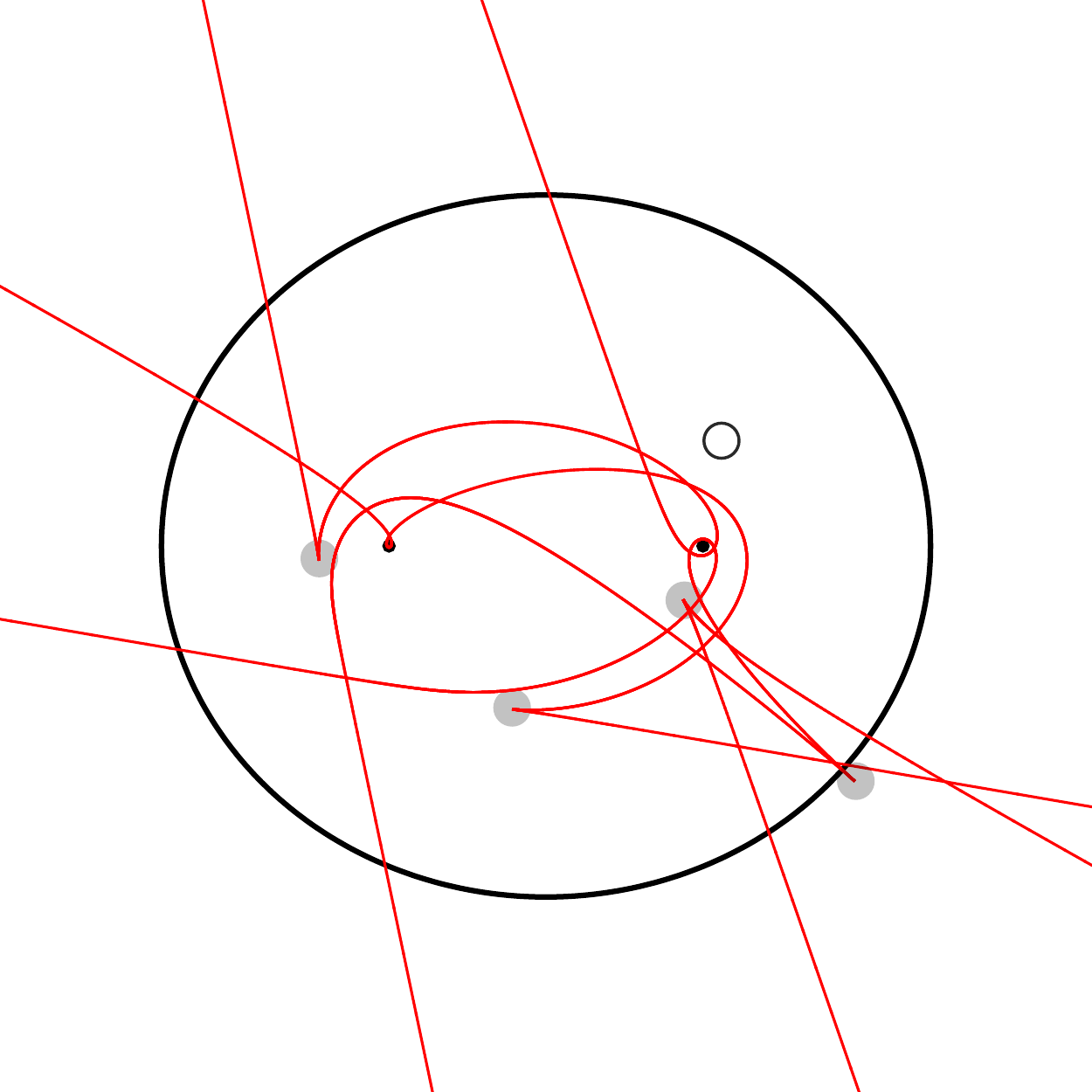}\quad
\includegraphics[width=.3\textwidth]{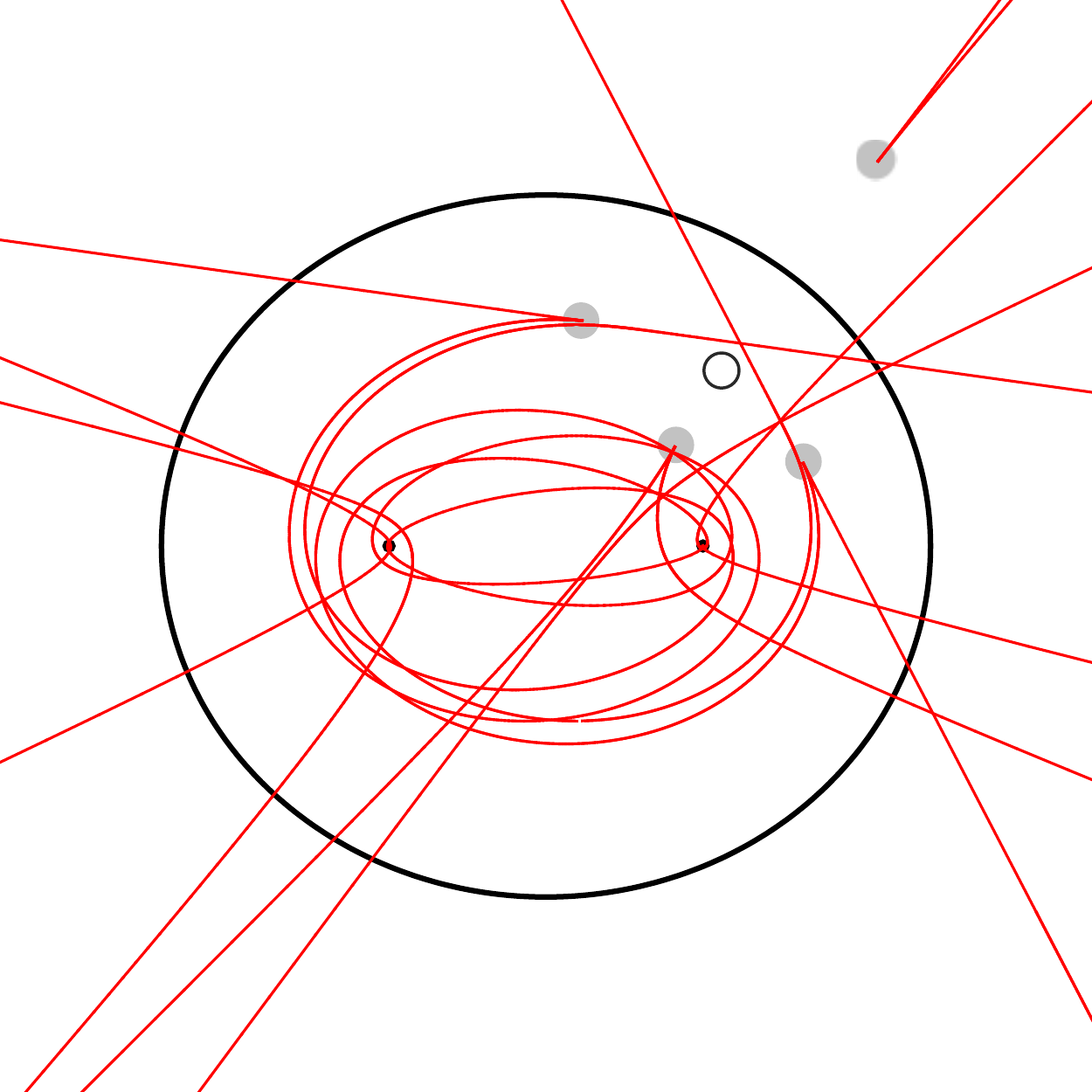}
\caption{The 2nd, 5th and 8th caustics by reflection in an ellipse, each with 4 cusps (marked by gray disks).}
\label{ellipse1}
\end{figure}

\begin{conjecture} \label{conj:two}
If $\gamma$ is not an ellipse then, for some choice of light source  inside $\gamma$  and some $n\ge 1$, the caustic by reflection $\Gamma_n$ has more than four cusps.
\end{conjecture}

Figure \ref{fig:4plus} shows caustics  with $>4$ cusps for non-elliptical $\gamma$ . 

\begin{figure}[ht]
\centering
\includegraphics[width=.4\textwidth]{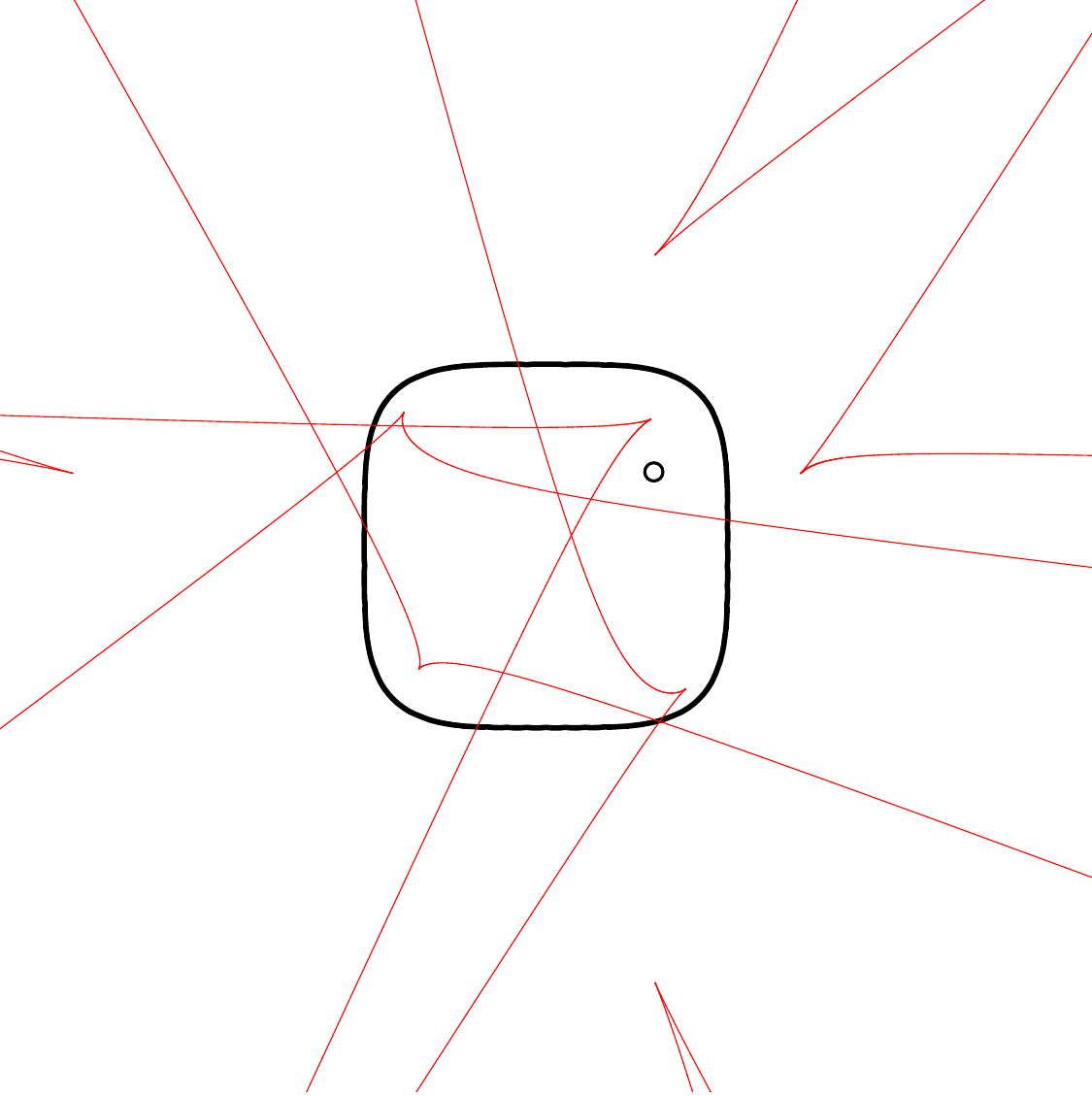}\hspace{.15\textwidth}
\includegraphics[width=.4\textwidth]{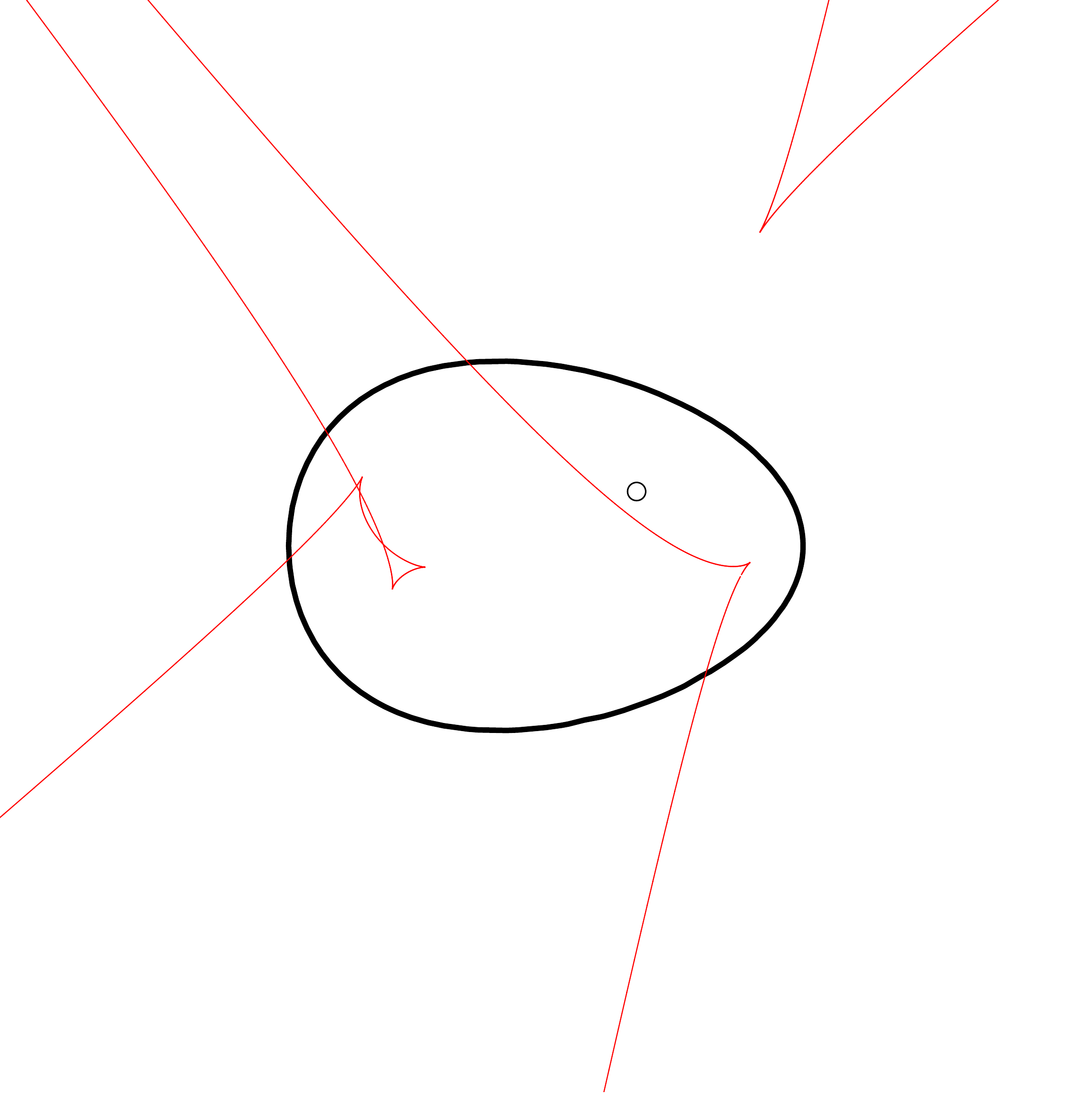}
\caption{First caustic by reflection with more than four cusps of non-elliptical ovals. Left: $x^4+y^4=1$, $O=(.6, .4)$. Right: $.5x^2+(1+.25x)y^2=1$, $O=(.5, .3)$ }
\label{fig:4plus}
\end{figure}

\subsection{Catacaustics} 
The first caustics by reflections, called {\em catacaustics}, are well studied. We give a brief summary of what is known about them, referring to \cite{Bo,BGG,Hu} and the literature cited in these articles.

A version of the string construction that recovers a billiard curve from a billiard caustic (see, e.g., \cite{Ta}) makes it possible to reconstruct the curve $\gamma$ from its first caustic by reflection $\Gamma_1$. This construction involves a parameter, the length of the string. See Figure \ref{string} (left).  

\begin{figure}[ht]
\centering
\def\svgwidth{.4\textwidth}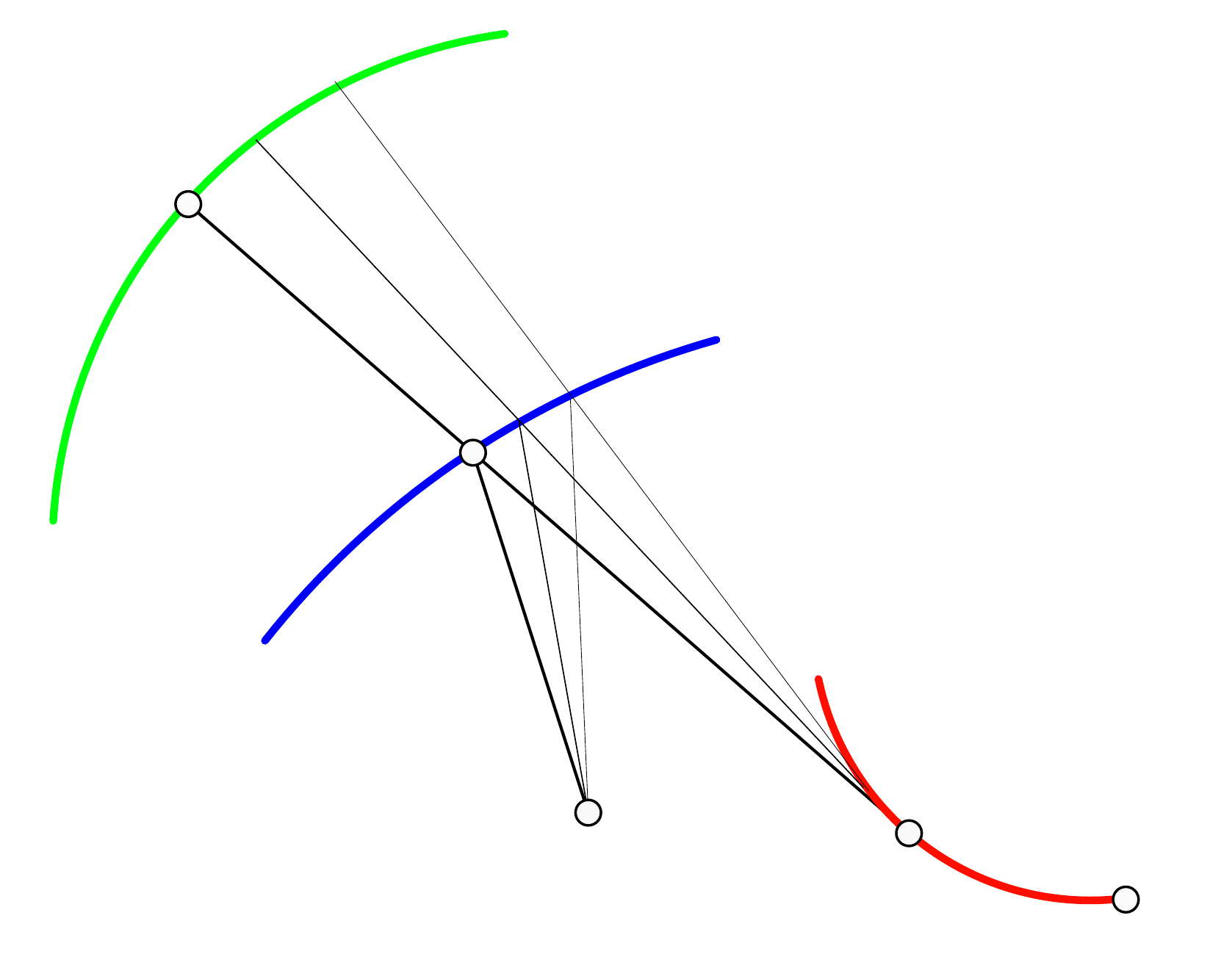
\hspace{.2\textwidth}
\def\svgwidth{.32\textwidth}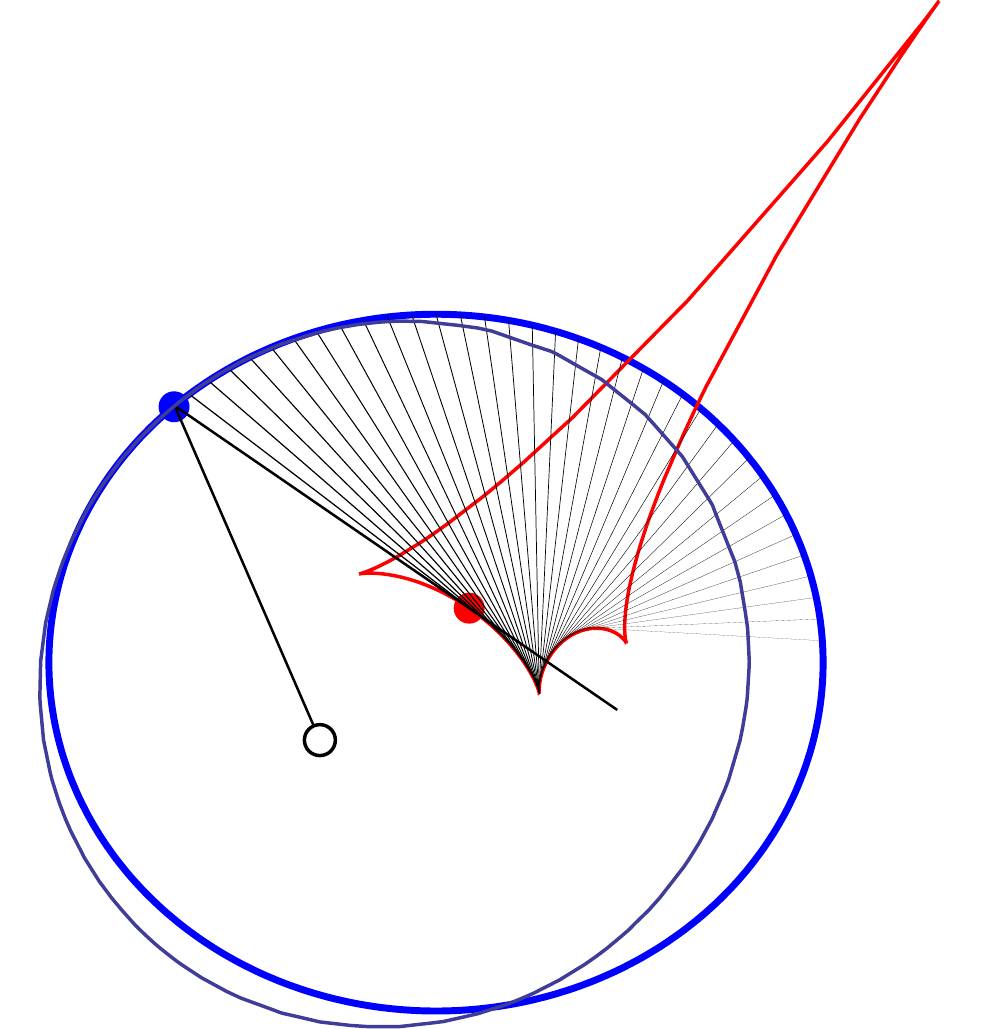
\caption{(a) The curve $\gamma$ is the locus of points $X$ such that $|OX|+ |XB| + \stackrel{\smile}{|BA|}=const$ (point $A$ is fixed on $\Gamma_1$). The point $Z$ is the reflection of $O$ in the tangency line to $\gamma$ at   $X$. The locus of points $Z$ is the orthotomic curve $\Delta_1$, orthogonal to $BZ$ and whose evolute is $\Gamma_1$. (b) The catacaustic $\Gamma_1$ (red) is the locus of 2nd foci $B$ of the osculating Kepler conic (gray) to the curve $\gamma$ (blue) with 1st foci at $O$.
}
\label{string}
\end{figure}

The orthotomic curve $\Delta_1$ is an involute of the catacaustic $\Gamma_1$,  see  Figure \ref{string} (left), and $\Gamma_1$ is the evolute of $\Delta_1$, that is, the envelope of its normals. The  cusps of $\Gamma_1$ correspond to the vertices of $\Delta_1$.
 It is known that when $\gamma$ is an ellipse and $O$ is not one of its foci then $\Gamma_1$ has 4 vertices \cite{BGG}. 

 A Kepler conic is a conic with one focus fixed at the origin $O$. The curve $\gamma$ has a Kepler conic that has 3-point contact with it at every point, see \cite{BT}. The locus of the second foci of these osculating Kepler conics is the first caustic $\Gamma_1$ -- this follows from the optical properties of conics (a ray from one focus reflects to another focus). It follows that the cusps of the catacaustics $\Gamma_1$ correspond to the points where the Kepler conics hyperosculate the curve $\gamma$. 
 
 \paragraph{Computer graphics and animations.} Most figures in this article were made
using the computer program Mathematica. They are complemented with some
animations on the web page   \url{https://www.cimat.mx/~gil/caustics/}.

\paragraph{Acknowledgements:}  
GB acknowledges  support from the Shapiro visiting program in Penn State and   CONACYT Grant A1-S-4588. ST was supported by NSF grant DMS-2005444.

\section{Background material} 

\subsection{The  phase cylinder and the billiard ball map}\label{sect:pc}

This section contains some standard material on mathematical billiards, see, e.g., \cite{Ta}.

Denote by $\L$ the space of oriented lines in $\R^2$. We use the `cylinder model' of $\L$, with  coordinates $(\alpha,p)$ defined as follows: $\alpha\in S^1=\R/2\pi\Z$ is the direction of the line and $p\in \R$ is the signed distance from the oriented line to the origin $O$ (which we choose to be the center of the initial beam of light). The sign of $p$ is defined by the right-hand rule, see Figure \ref{lines}.  Thus $\L$ is an infinite cylinder.

\begin{figure}[ht]
\centering
\def\svgwidth{\textwidth}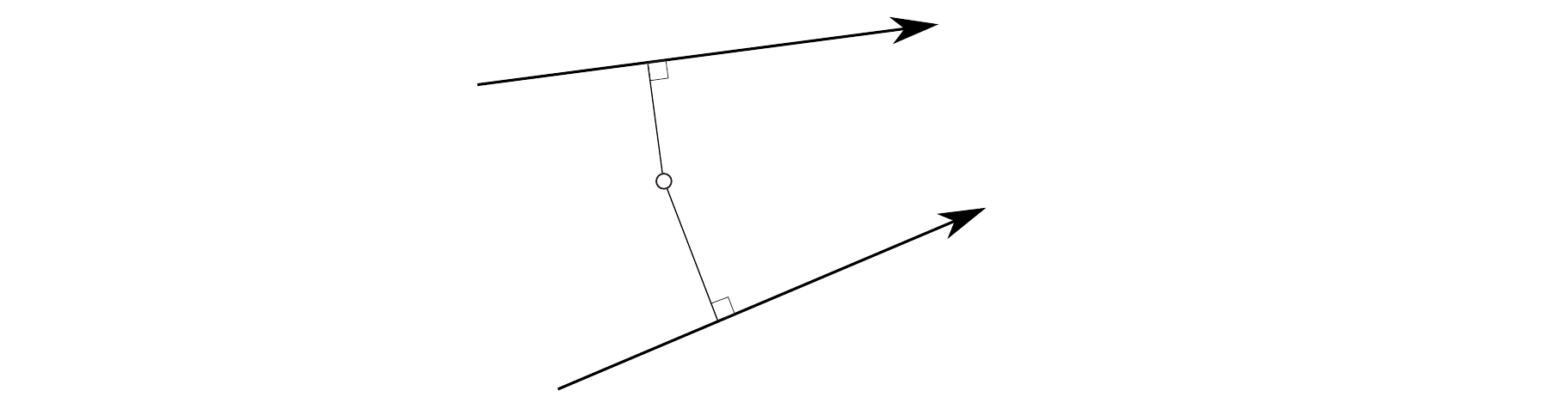
\caption{The coordinates $(\alpha,p)$ of an oriented line in $\R^2$.}
\label{lines}
\end{figure}

The space $\L$ of oriented lines in $\R^2$ admits an area form, unique up to scale, invariant under the Euclidean group action. In coordinates, this area form is $\omega=d\alpha\wedge dp.$

The {\em phase cylinder} of the billiard system inside an oval $\gamma\subset\R^2$ is  the set $M\subset \L$ of oriented lines intersecting  $\gamma$. It is a bounded cylinder whose two  boundary components  correspond to the lines tangent to $\gamma$, one component for each orientation.  The ``equator'' $p=0$ corresponds to the lines through $O$, see Figure \ref{fig:pc}.

\begin{figure}[ht]
\centering
\def\svgwidth{.3\textwidth}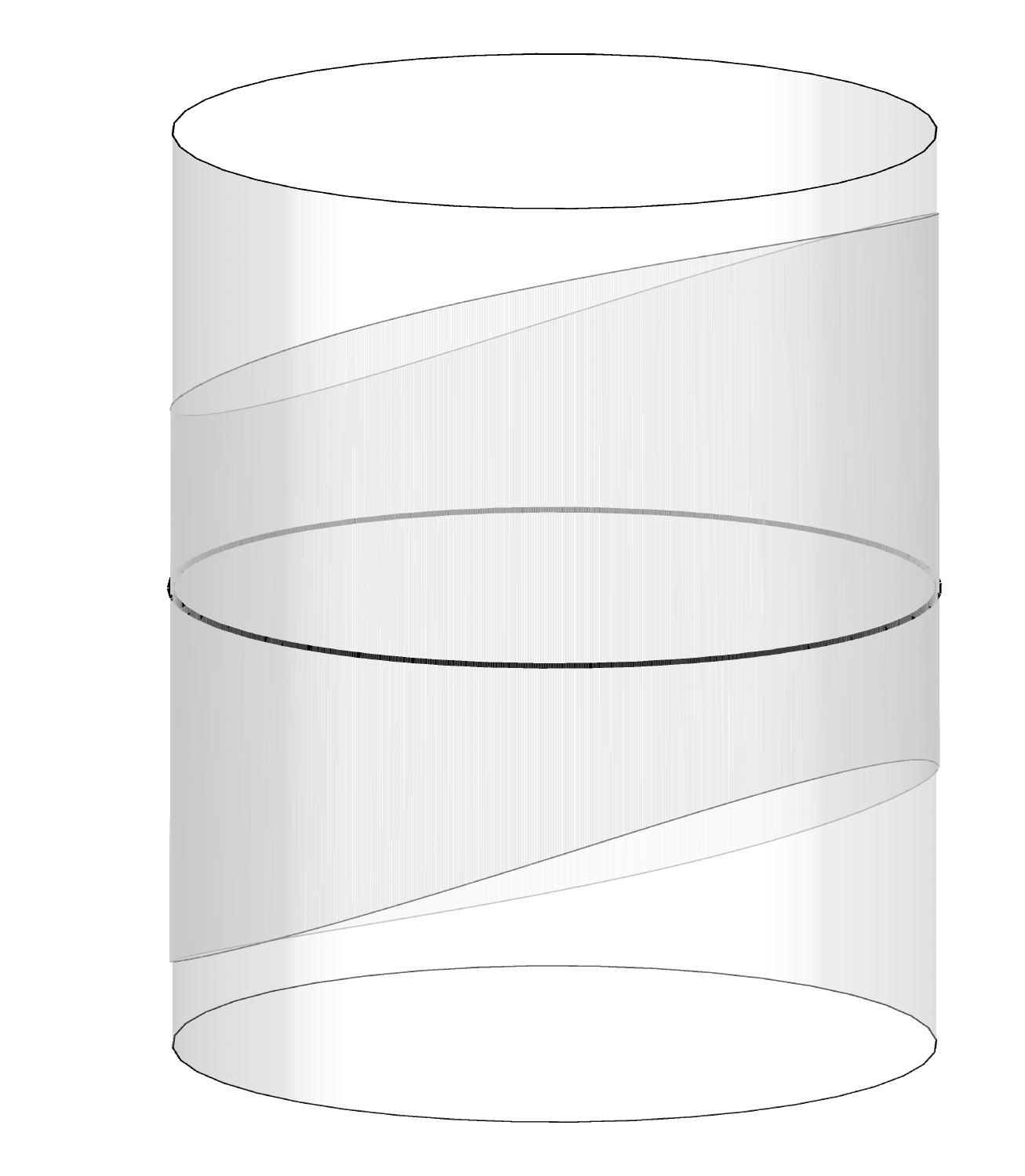
\hspace{.2\textwidth}
\def\svgwidth{.4\textwidth}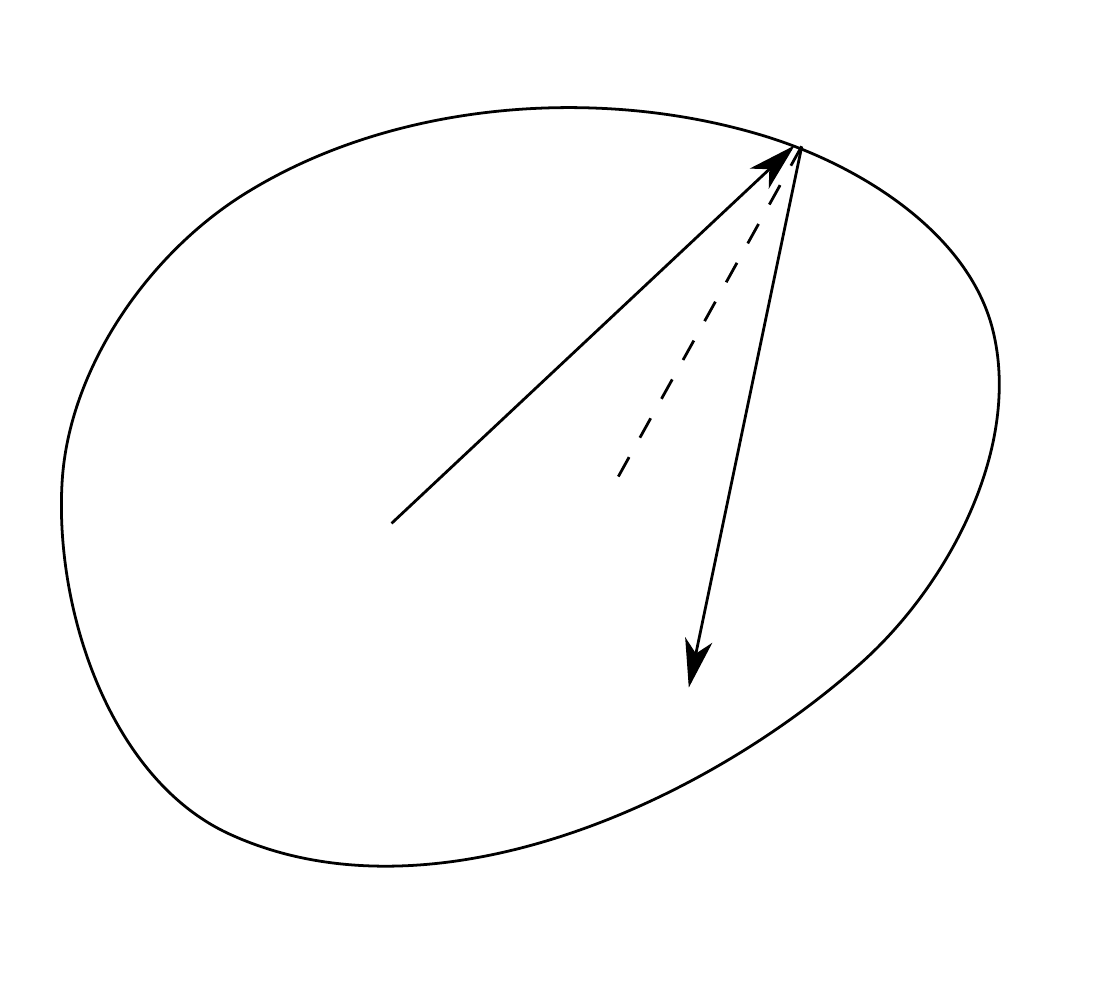
\caption{Left: the phase cylinder $M\subset \L$. Right: the billiard ball map $T:M\to M$. }
\label{fig:pc}
\end{figure}

The billiard ball map $T:M\to M$, sending  an incoming ray to the reflected one, is an area preserving transformation, that is, $T^*\omega=\omega$.  
Since $\omega=-d(pd\alpha)$, the differential 1-form $T^*(pd\alpha)-pd\alpha$ is closed. In fact, more is true: as we will now show, it is {\it exact}, that is, $T^*(pd\alpha)-pd\alpha=dF$ for some function $F:M\to\R$. An example of an area preserving, but non-exact, map   is   $(\alpha,p)\mapsto (\alpha, p+1).$ 

\begin{proposition}\label{prop:exact} The billiard ball map $T:M\to M$ is exact. 
\end{proposition}

In order to  prove Proposition \ref{prop:exact},  consider another description of the phase cylinder, as the set of unit vectors with a foot point on $\gamma$, pointing inwards, the initial position and velocity of the billiard ball. These unit vectors are in one-to-one correspondence with the oriented lines that they generate.   Let $\gamma(t)$ be an arc length  counterclockwise parameterization and  $\varphi$ be the angle between the tangent $\gamma'(t)$ and the unit vector. See Figure \ref{fig:ellip}a.

\begin{figure}[ht]
\centering
\def\svgwidth{.4\textwidth}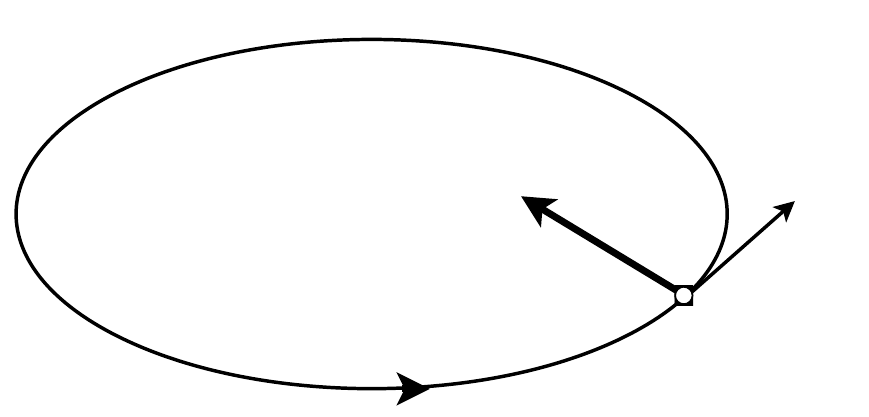\hspace{.15\textwidth}
\def\svgwidth{.39\textwidth}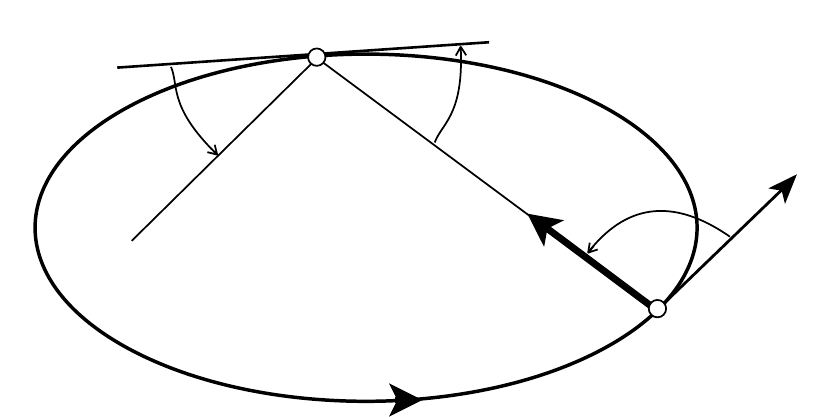
\caption{(a) The coordinates $(t, \varphi)$ on $M$ and their relation to $(\alpha,p).$ (b) The generating function $L$ of the billiard ball map $T$.
}
\label{fig:ellip}
\end{figure}

Consider the differential 1-form $\cos\varphi\, dt$. Let $L=|\gamma(t_1)-\gamma(t)|$ be the distance between the intersection points of a line   with $\gamma$. See Figure \ref{fig:ellip}b.

\begin{lemma} \label{lm:exact} 
$T^* (\cos\varphi\ dt)-\cos\varphi\ dt = dL$.
\end{lemma}

\begin{proof}
One has: $T(t,\varphi)=(t_1, \varphi_1)$ and 
$$
\frac{\partial L(t,t_1)}{\partial t} = - \cos\varphi,\ \frac{\partial L(t,t_1)}{\partial t_1} =  \cos\varphi_1,
$$
that is,
$
dL= \cos\varphi_1\ dt_1 - \cos\varphi\ dt,
$
as needed. \end{proof}

Two differential 1-forms are cohomologous if their difference is the differential of a function.

\begin{lemma} \label{lm:cohom}
The 1-form $p\, d\alpha$ is cohomologous to $\cos\varphi\ dt$.
\end{lemma}

\begin{proof} 
Using complex notation, see Figure \ref{fig:ellip}a, one has
$$
e^{i(\alpha-\varphi)}=\gamma'(t),\ p=\det( \gamma(t), e^{i\alpha} ).
$$
Differentiating  these equations, we get
$$
d\alpha\equiv d\varphi\ (\mbox{mod }dt),\ dp\equiv \sin\varphi\, dt\ (\mbox{mod }d\alpha),
$$ so 
$$
d\alpha\wedge dp=\sin\varphi\, d\varphi \wedge dt=-d(\cos\varphi\ dt),
$$
hence $p d\alpha - \cos\varphi\ dt$ is closed. 

To show that $p\, d\alpha - \cos\varphi\,dt$ is exact it suffices to show that its integral along a non-contractible closed curve in $M$ vanishes. 
As such a curve take $C$, the boundary component of the phase cylinder $M$ given by $\varphi =0$. Clearly, $\int_C \cos\varphi\,dt$ equals the perimeter of $\gamma$.

On the other hand, by the Cauchy-Crofton formula, this perimeter equals $\pi$ times the average length of the orthogonal projection of $\gamma$ on a line, that is,  $\int_C p d\alpha$. 
This implies the result. 
\end{proof}

\noindent {\em  Proof  of Proposition \ref{prop:exact}.} Lemma \ref{lm:exact} shows that $T$  preserves the integral $\int_C\cos\varphi\ dt$, and Lemma \ref{lm:cohom} shows that $\int_C\cos\varphi\ dt=\int_C p\,d\alpha$, hence $T$ preserves the integral $\int_C p\,d\alpha$, i.e., it is exact.  \qed

\mn 

The {\em signed area} enclosed by an oriented  closed curve $C$ in $\L$ is the line integral $\int_C pd\alpha$. Clearly, the curve $C_0 $ representing the  initial beam of light encloses zero area. Since $T$ is an exact map, so does $C_n=T^n(C_0)$. See Figure \ref{fig:3beams}. 

\begin{figure}[ht]
\centering
\def\svgwidth{\textwidth}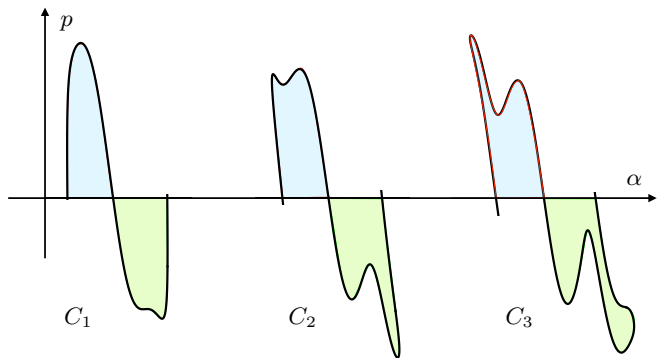
\caption{The succesive iterates $C_n=T^n(C_0)$, for $n=1,2,3$,  drawn on the phase cylinder $\L$ (spread flat). Each is a simple closed smooth curve of 0 signed area (blue area = green area). The billiard table is the ellipse $4x^2/5 +y^2=1$ and the source is $(.6, .2)$.}
\label{fig:3beams}
\end{figure}

\subsection{Contact elements, Legendrian knots, and wave fronts}

A {\em contact element} in the plane is a pair consisting of a point and a line through it. A coorientation of a contact element is a choice of one of the sides of the line. More conceptually, the space of cooriented contact elements is the spherization of the cotangent bundle $ST^*\R^2$: assign to a covector $\eta$ its kernel, a tangent line, and define coorientation by choosing the side on which $\eta$ is positive.  

The space of contact elements carries a contact structure, a 2-dimensional distribution defined by the ``skating condition": the foot point may move along the line, and the line may rotate about the foot point.
 Let $(x,y)$ be the standard  coordinates in $\R^2$ and $\theta$  the angle between the positive $x$-axis and the direction of the line;  then the contact distribution  is the kernel of the 1-form
$
\sin\theta\, dx-\cos\theta\, dy.
$

A smooth curve in $ST^*\R^2$ that is tangent to the contact distribution is called {\em Legendrian}. Its projection to the plane  is a {\em wave front}, a curve that may have singularities, generically semicubical cusps, but that has a tangent line at every point. Conversely, such a curve has a unique lift to the space of contact elements as a Legendrian curve.  

We introduce   coordinates $(\alpha,p,z)$ in $ST^*\R^2$, where $(\alpha,p)\in T^*S^1$ are the coordinates of the orthogonal line at the foot point $A$, and $z$ is the (signed) distance  of the line to the origin $O$.  See Figure \ref{contact}.

\begin{figure}[ht]
\centering
\def\svgwidth{\textwidth}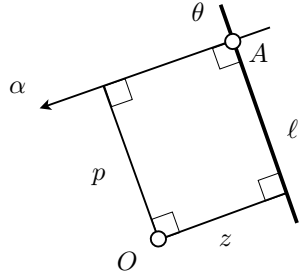
\caption{Coordinates $(\alpha,p,z)$ on the space $ST^*\R^2$ of cooriented contact elements in $\R^2$. With a cooriented line $\ell$ through $A$ we associate the (signed) distance $z$ to $O$ and the coordinates $(\alpha,p)$ of the perpendicular oriented line at $A$, pointing  to the positive side of $\ell$. 
}
\label{contact}
\end{figure}

This defines an identification of $ST^*\R^2$ with $J^1S^1$, the  space of 1-jets of functions $f:S^1\to\R$,  where $z=f(\alpha)$ and $p=f'(\alpha).$ On $J^1S^1$ there is a standard  contact form $dz-pd\alpha$ (the 1-jets of functions $z=f(\alpha)$ are Legendrian curves).

\begin{lemma} The identification $ST^*\R^2=J^1S^1$ is a contactomorphism. 
\end{lemma}
\begin{proof} Using the notation of Figure \ref{contact}, 
$$
\alpha=\theta+\pi/2,\ p=\det(A, e^{i\alpha}),\ z=\det(A, e^{i\theta}),
$$
 so 
 $$
 dz-pd\alpha=\det(dA, e^{i\theta})=\sin\theta dx-\cos\theta dy,
 $$ 
 as claimed.
 \end{proof}

We have two projections 
\[
\begin{tikzcd}
& ST^*\R^2\arrow[dl,swap,"\pi_1"] \arrow[dr,"\pi_2"] \\
\R^2  && \L
\end{tikzcd}
\]
where  $\pi_1$ maps a cooriented contact element $(A,\ell)$ to $A$ and  $\pi_2$ maps it to the  line through $A$, orthogonal to $\ell$, oriented towards its positive side. 

In coordinates, $\pi_2:
(\alpha,p,z)\mapsto (p,\alpha)$ (``forgetting $z$''). The fibers of  $\pi_1$ are Legendrian, spanned by the vector field $\partial_\alpha-z\partial_p+p\partial_z$, while the fibers of  $\pi_2$, spanned by the vector field $\partial_z$,  are transverse to the contact distribution. Hence $\pi_2$ projects a smooth Legendrian curve in $ST^*\R^2$ to a smooth curve in $\L$. 

Conversely, let $C$ be a   closed  curve in $\L$. We want to lift it via $\pi_2:ST^*\R^2\to \L$ to a Legendrian curve $\widetilde C\subset ST^*\R^2$. Since the contact distribution on $ST^*\R^2$ is transverse to the fibers of $\pi_2$, once the initial point of the lifting is chosen, the lifting is uniquely determined, but it may fail to close up. 

\begin{lemma} \label{lm:front}
The lifted curve $\widetilde C$ is closed if and only if $\int_C p d\alpha =0$: the curve $C$ encloses zero signed area.
\end{lemma} 

\begin{proof} The curve
$\widetilde C$ is closed if and only if the value of the third coordinate $z$ is the same at the endpoints. Since $dz=p d\alpha$ along a Legendrian curve, the values of $z$ at the endpoints are equal if and only if  $\int_C p d\alpha =0$.
\end{proof}

The curve $C\subset \L$ defines a 1-parameter family of oriented lines.
The projection of the lifted Legendrian curve $\widetilde C\subset ST^*\R^2$ to $\R^2$ is a wave front $\Delta$ that is orthogonal to this family of lines and is cooriented by their directions. If a closed   front $\Delta$ exists, i.e., $C$ encloses zero signed area, then there exists a whole 1-parameter family of fronts that are equidistant from each other. This non-uniqueness corresponds to the choice of the initial point of the lifted curve $\widetilde C$. 

The situation is the same as in the familiar  relation between evolutes and involutes: for an involute of a closed curve to close up it is necessary and sufficient for the curve to have zero signed length (the sign changes after each cusp), and the equidistant family of curves share their normals, and hence their evolutes.

\subsection{Vertices of wave fronts and Legendrian isotopies} 

A {\em vertex} of a plane curve is an extremum of its curvature or, equivalently, a cusp of the evolute, the envelope of its normals. The notion of vertex extends to cooriented wave fronts: the curvature at cusps is infinite, changing from $-\infty$ to $\infty$ (so  cusps are not vertices).

The classical 4-vertex theorem asserts that a simple closed convex curve has at least four vertices.  Let $\Delta$ be a cooriented wave front whose Legendrian lift to $ST^*\R^2$  is embedded, i.e., is a {\em Legendrian knot}. V. Arnold conjectured \cite{Ar,Ar1} that if this Legendrian knot is  homotopic as a Legendrian knot to the Legendrian lift of a circle, then $\Delta$ has at least four vertices. This conjecture was proved by Chekanov and Pushkar \cite{ChP} using Legendrian knot theory. 

A generic regular homotopy of a cooriented wave front is a composition of a number of moves, similar to the Reidemeister moves in knot theory, see Figure \ref{moves}, borrowed from \cite{ChP}. The first five moves are isotopies of the respective Legendrian knot, but the ``dangerous" self-tangency with coinciding coorientations correspond to self-intersection of the Legendrian lifted curve and changes the Legendrian knot type.

\begin{figure}[ht]
\centering
\includegraphics[width=\textwidth]{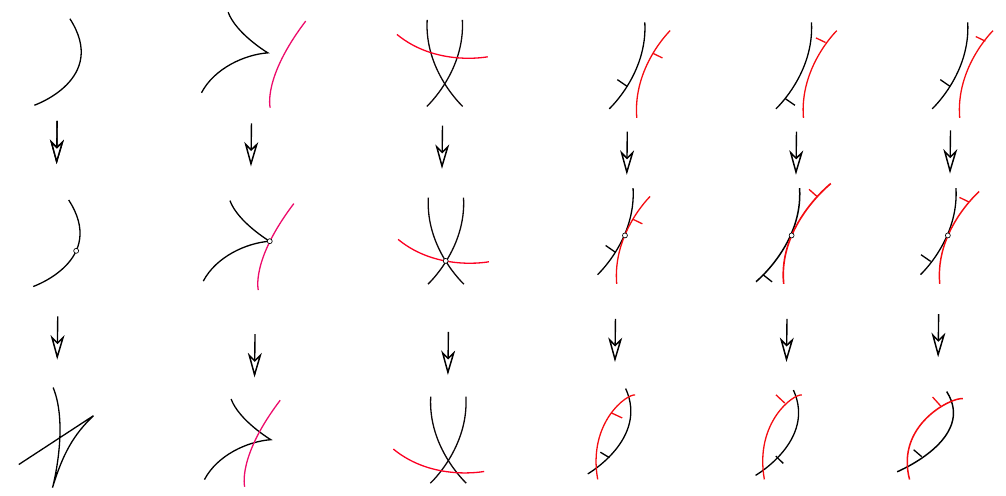}
\caption{Generic ``perestroikas" of cooriented wave fronts.   
}
\label{moves}
\end{figure}

For example, the curve on the left of Figure \ref{curves} has only two vertices, but the curve on the right is Legendrian isotopic to a circle, therefore  it has at least four vertices no matter how one draws it. Thus these curves are not Legendrian isotopic. On the other hand,  the Whitney winding number of both curves is one, hence they are regularly isotopic.

\begin{figure}[ht]
\centering
\includegraphics[height=.23\textwidth]{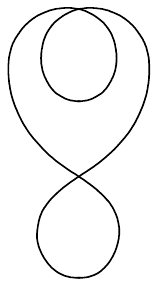}
\hspace{.2\textwidth}
\includegraphics[height=.23\textwidth]{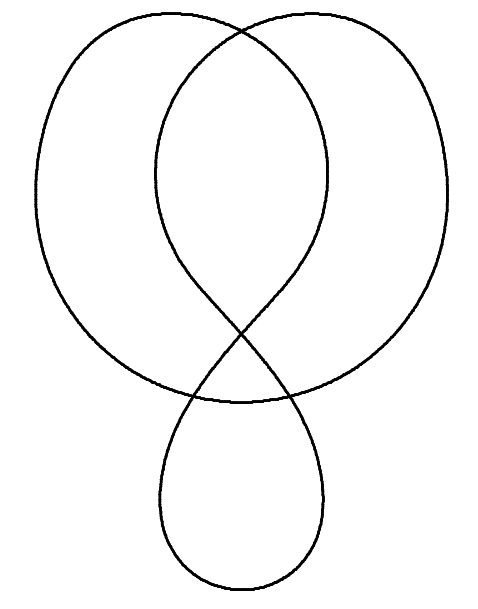}
\caption{Left: only two vertices; right: at least four vertices.   
}
\label{curves}
\end{figure}

\subsection{Summary}
With each ``beam'' of light rays   (a 1-parameter family of oriented lines in $\R^2$) we have associated four curves, 
$$ C\subset \L,\qquad  \widetilde C\subset ST^*\R^2,\qquad \Delta\subset\R^2, 
\qquad\Gamma\subset \RP^2.
$$
The correspondences between the cusps, vertices and inflection points on these curves is depicted in Figure \ref{fig:co}. 
\begin{figure}[ht]
\centering
\def\svgwidth{.6\textwidth}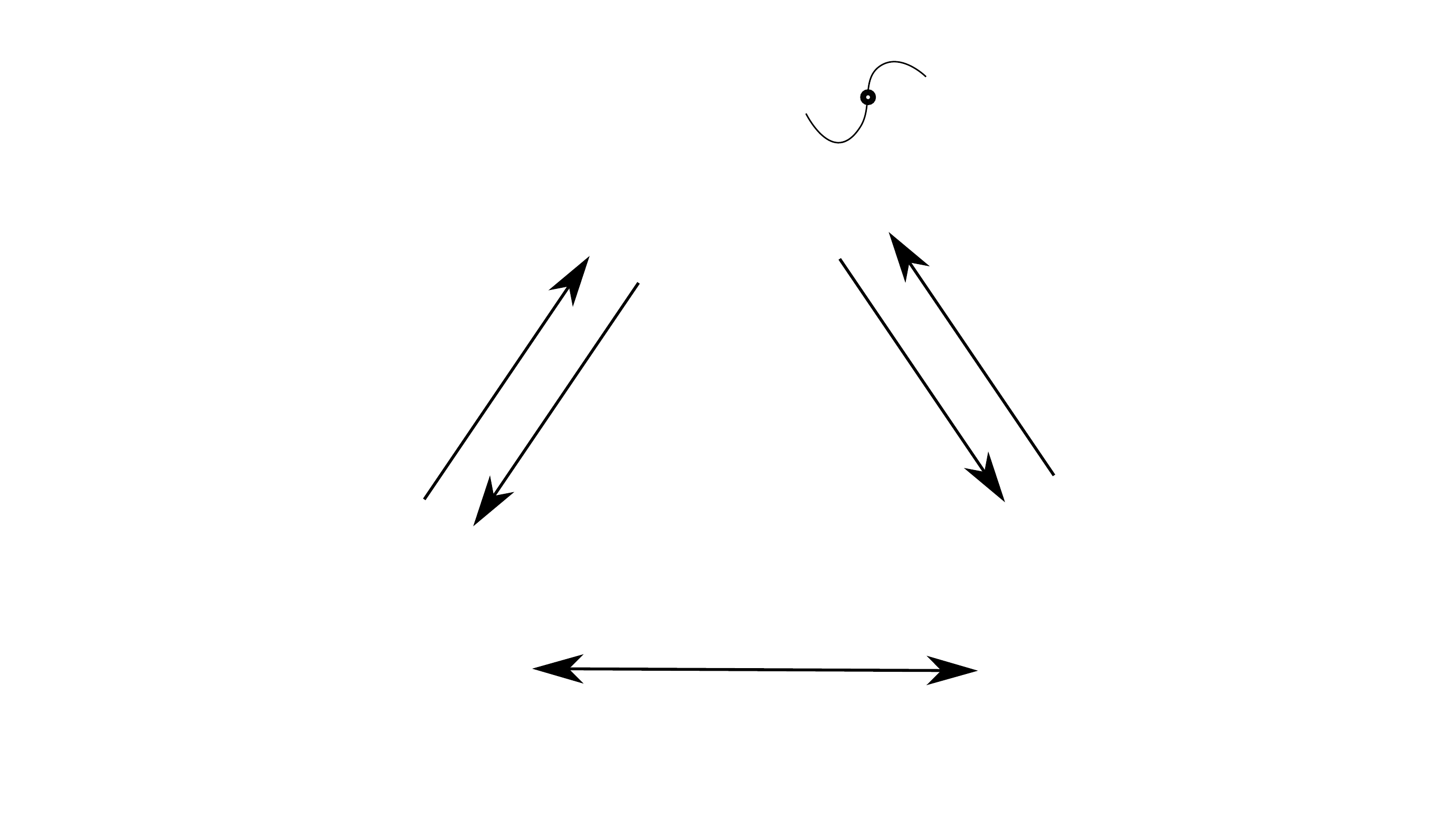
\caption{
}
\label{fig:co}
\end{figure}

\section{Proofs of Theorem \ref{thm:main}}

\paragraph{First proof.}  Cusps of $\Gamma_n$ correspond, by projective duality, to inflection points of $C_n$. These  inflections points  are  3-point contacts of $C_n$ with the curves in $\L$ corresponding to the 1-parameter families of lines passing through a fixed point. If the point is $(a,b)\in\R^2$, then the respective curve in $\L$ is
the graph of the first harmonic
$$
p=a\sin \alpha - b \cos \alpha,
$$
that is, it is an ellipse obtained as the intersection of the cylinder $\L$ with a plane through the origin. Note that this graph encloses zero signed area.

Consider the central projection of this cylinder to the unit sphere. This projection sends the graphs of the first harmonics to great circles. 

Since $C_n$ encloses zero signed area, it intersects every graph of the first harmonics. It follows that $\bar C_n$, the image of the curve $C_n$, is a smooth spherical curve that is not contained in any hemisphere. In particular, the convex hull of $\bar C_n$ contains the origin.

The (geodesic)  inflections of $\bar C_n$ in the standard metric of the sphere are its 3-point contacts with great circles. By the Segre theorem mentioned earlier, $\bar C_n$ has at least four inflections. Therefore so does $C_n$. \qed

\paragraph{Second proof.} Following \cite{An}, one can use the curve shortening flow to prove that the curve $C_n$ has at least four inflections. 
Recall that under the curve shortening flow, each point of the curve moves in the normal direction with the speed equal to the curvature; see
\cite{GH,Gr} and the book \cite{CZ}. 

Equip $\L$ with the flat Riemannian metric $d\alpha^2+dp^2$ and apply the curve shortening flow to $C_n$. Let $C_n(s)$ be arclength parametrization, then the flow is given by the partial differential equation $C_t=C_{ss}$.  

A variation of the standard proof shows that the evolution is defined for all $t\geq 0$, deforming $C_n$  through embedded curves,  shrinking it to a horizontal curve $p=const$, which is a closed geodesic.
A version of the maximum principle implies that the number of inflections does not increase during this evolution, see \cite{An}. This is illustrated in Figure \ref{infl}. 

\begin{figure}[ht]
\centering
\includegraphics[height=.35\textwidth]{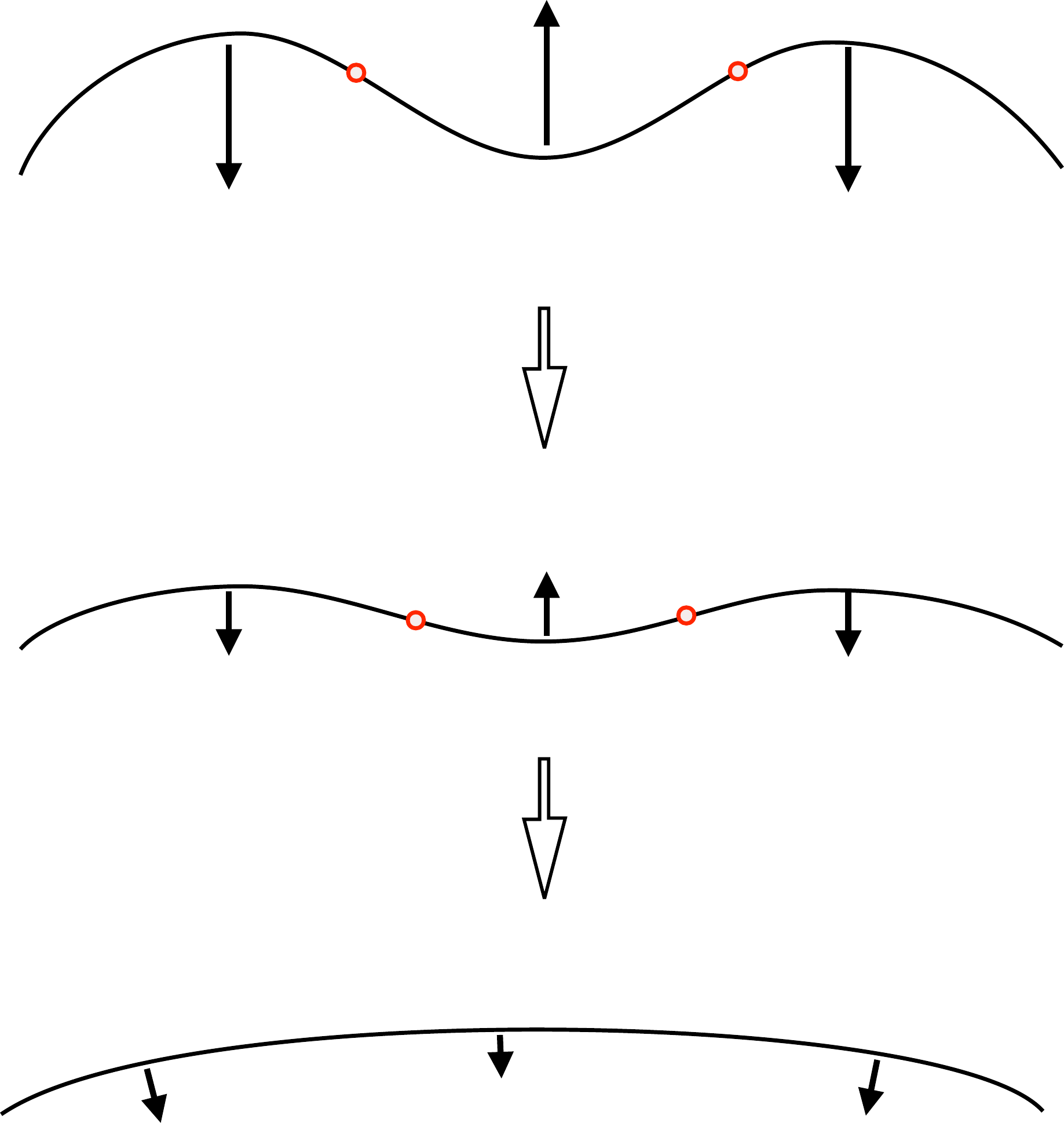}
\caption{In the curve shortening flow, two nearby inflection points may cancel each other, but they cannot appear on a convex arc.   
}
\label{infl}
\end{figure}

Next, a version of Lemma 3.1.7 in \cite{GH}  or Lemma 1.10 of \cite{Gr} shows that the evolving curves enclose zero signed areas.

\begin{lemma} \label{lm:flow}
The curve shortening flow   $C_t=C_{ss}$  preserves the signed area $\int_C p d\alpha.$ \end{lemma}

\begin{proof} 
Let $(\alpha(s),p(s))$ be arc length parameterization of a non-contractible  closed curve in $\L$, so that $\alpha_s^2+p_s^2=1$. Then its time evolution under the curve-shortening flow is given by
$
(\alpha_t,p_t)=kN=k(-p_s,\alpha_s),
$
where the subscript denotes the derivative, $k$ is the curvature, and $N$ is the unit normal. It follows that
$$
\frac{d \int p d\alpha}{dt} = \int [k \alpha_s^2 - p (kp_s)_s] ds = \int [k \alpha_s^2+kp_s^2] ds = \int k(s) ds,
$$
where the second equality is due to integration by parts. 

It remains to note that the total curvature of a closed curve that goes around the cylinder equals zero. 
\end{proof} 

 As $C_n$ approaches  $C_0$, it is given by the graph of a function $p=F(\alpha)$. The inflection points of this graph are the points where it is tangent to 2nd order to  the graphs of functions of the form $h(\alpha)=a\cos(\alpha)+b\sin(\alpha) .$ 
 
 For each $\alpha\in S^1$, one can find unique $a,b$ such that 
 $F(\alpha)=h(\alpha)$, $ F'(\alpha)=h'(\alpha).$ Since $h''+h=0,$ the equation $F''(\alpha)=h''(\alpha) $ holds if and only if  $F''(\alpha)+F(\alpha)=0.$ So the inflection points of $F$ are  the zeros  of the function  $G:=F''+F.$  
 
 To conclude that $G$ has no less than four zeros, apply the Sturm-Hurwitz theorem that states that the number of zeros of a $2\pi$-periodic function is not less than the number of zeros of its first non-trivial harmonic, see, e.g., \cite{Ar,Ar1}. 
 
 Since the differential operator $d^2+1$ preserves the order of Fourier terms and kills 1st order terms, $G$ has no first harmonics. Since the curve encloses zero signed area, $F$ has zero constant term, and so does $G$, as needed. \qed 
 
 \begin{remark}
 {\rm The Sturm-Hurwitz theorem has many proofs, see Section 8.1 of \cite{OT}. Interestingly, one of them, due to G. Polya, makes use of the heat equation, a close relative  of the curve shortening flow. 
 }
 \end{remark}
 
\paragraph{Third proof.} 
This argument relies on the correspondence between the cusps of $\Gamma_n$ and the vertices of $\Delta_n$, its normal front. 

Since $C_n$ encloses zero signed area, Lemma \ref{lm:front} implies that it admits a Legendrian lift $\widetilde C_n\subset ST^*\R^2$, and its  projection to $\R^2$ is a closed curve, possibly with cusps, which is normal to the rays of $C_n$.

A homotopy of the curve $C_n$ to $C_0$ in the class of smooth closed embedded curves that enclose zero signed area induces a Legendrian isotopy between the Legendrian knots   $\widetilde C_n$ and $\widetilde C_0$. (Such a homotopy is provided by the curve shortening flow but, unlike the second proof, one can use any other homotopy for this purpose). 

Now our ``black box", the Pushkar-Chekanov  theorem \cite{ChP}, implies that $\Delta_n$ has at least four vertices. \qed

\section{Variations}

In this last section we briefly mention other variants of our result.

\paragraph{Spherical and hyperbolic geometry.}
One can extend  Theorem \ref{thm:main} to geodesically convex billiards in  spherical and hyperbolic geometries (the former lie in one hemisphere). 

The space of oriented geodesics (great circles) in $S^2$ is identified with $S^2$ itself via the usual equator/pole correspondence. The phase cylinder $M\subset S^2$ parametrizes oriented geodesics intersecting $\gamma$. The billiard ball map   preserves the standard area form on $S^2$, and the curve $C_n=T^n(C_0)$ bisects the area of the sphere. 

According to Arnold's ``tennis ball theorem" \cite{Ar} (which also follows from the theorem of Segre), a smooth closed embedded spherical curve that bisects the area has at least four inflections. As before,  the curve $C_n$ is dual to the caustic $\Gamma_n$, hence the latter has at least four cusps.

In the hyperbolic case, consider  the hyperboloid model  $H^2=\{x^2+y^2-z^2=-1, \ z>0\}$, with the Minkowski metric $dx^2+dy^2-dz^2$ restricted to $H^2$. An oriented geodesic is given by intersecting $H^2$  with  an oriented plane through the origin, the orthogonal complement (in the Minkowski sense) of a space-like unit vector. 

These vectors comprise a hyperboloid of one sheet $H^{1,1}=\{x^2+y^2-z^2=1\}$, equipped with the area form induced from the ambient Minkowsky space (this area form is invariant under the group of motions  $SO(2,1)$). 
The phase cylinder $M\subset H^{1,1}$ corresponds to geodesics intersecting $\gamma$, and the billiard ball map is exact area preserving. 

To show that the curve $C_n$ has at least four inflections, we use the same argument as in our first proof of Theorem \ref{thm:main}: centrally project the hyperboloid to the sphere. This projection takes inflections to inflections, and the image of $C_n$ contains the origin in its convex hull. Then the Segre theorem implies the result.

\paragraph{Projective billiards.} For any convex curve $\gamma\subset\R^2$ with a transverse vector field $v$ along it one can define the projective billiard map $T:M\to M$ \cite{TaE,TaG}. The reflection law is as follows. Consider an incoming ray  at a point $x\in\gamma$ in the  direction $u$, decompose $u=u_1+u_2$,  where $u_1$ is tangent to $\gamma$ at $x$ and $u_2$ is a multiple  of $v(x)$. Then the  outgoing ray passes through $x$ in the direction $u_1-u_2$. Equivalently,  the tangent line, the transverse line, the incoming, and the outgoing ones, form a harmonic quadruple of lines.

If the transverse field consists of the normals, one has the usual  law ``the angle of incidence equals the angle of reflection". 

If $\gamma$ is an origin-centered ellipse and the transverse field $v$ is given by the gradient of a homogeneous function of two variables, then the projective billiard ball map is again exact area preserving. The area form on the phase space $M$ is the same as the one on the space of oriented geodesics in the hyperbolic plane, but this time one considers the projective, or Cayley-Klein, model of hyperbolic geometry in the interior of the ellipse $\gamma$. The total area of $M$ is infinite in this case.

As before, Theorem \ref{thm:main} holds: see Figure \ref{fig:proj} for an illustration.

\begin{figure}[ht]
\centering
\includegraphics[width=.7\textwidth]{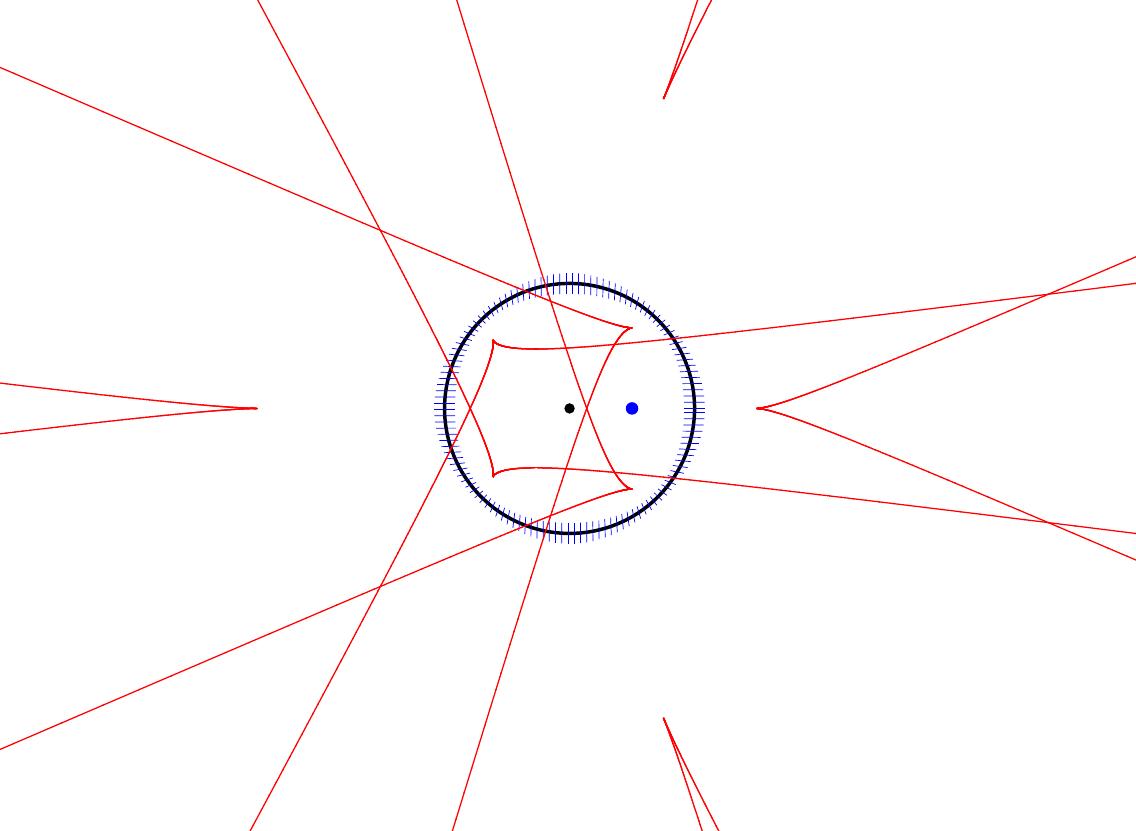}
\caption{The 1st  caustic by  reflection, showing 8 cusps,  in a projective billiard system with a circular table and the  exact transverse field $v= \nabla (x^4+y^4)$. 
}
\label{fig:proj}
\end{figure}


\paragraph{Other initial beams.} 
Finally, we note that Theorem \ref{thm:main}, along with its proofs, extends to some other initial beams of light. For example, one may consider a 1-dimensional source, an oval that lies inside $\gamma$ and that emanate rays of light in the outward normal directions.

\end{document}

%% file: 3caustics.pdf_tex
\begingroup%
  \makeatletter%
  \providecommand\color[2][]{%
    \errmessage{(Inkscape) Color is used for the text in Inkscape, but the package 'color.sty' is not loaded}%
    \renewcommand\color[2][]{}%
  }%
  \providecommand\transparent[1]{%
    \errmessage{(Inkscape) Transparency is used (non-zero) for the text in Inkscape, but the package 'transparent.sty' is not loaded}%
    \renewcommand\transparent[1]{}%
  }%
  \providecommand\rotatebox[2]{#2}%
  \ifx\svgwidth\undefined%
    \setlength{\unitlength}{319.07399902bp}%
    \ifx\svgscale\undefined%
      \relax%
    \else%
      \setlength{\unitlength}{\unitlength * \real{\svgscale}}%
    \fi%
  \else%
    \setlength{\unitlength}{\svgwidth}%
  \fi%
  \global\let\svgwidth\undefined%
  \global\let\svgscale\undefined%
  \makeatother%
  \begin{picture}(1,0.81742792)%
    \put(0.13215889,0.02739876){\color[rgb]{0,0,0}\makebox(0,0)[lb]{\smash{\s$\gamma$}}}%
    \put(0.35526647,0.26830492){\color[rgb]{0,0,0}\makebox(0,0)[lb]{\smash{\s$\Gamma_1$}}}%
    \put(0.18935984,0.44167382){\color[rgb]{0,0,0}\makebox(0,0)[lb]{\smash{\s$\Gamma_2$}}}%
    \put(0.71791978,0.46916641){\color[rgb]{0,0,0}\makebox(0,0)[lb]{\smash{\s$\Gamma_3$}}}%
    \put(0,0){\includegraphics[width=\unitlength,page=1]{3caustics.pdf}}%
    \put(0.85915598,0.54167263){\color[rgb]{0,0,0}\makebox(0,0)[lb]{\smash{\s$O$}}}%
  \end{picture}%
\endgroup%

%% file: circle_caustics1c.pdf_tex
\begingroup%
  \makeatletter%
  \providecommand\color[2][]{%
    \errmessage{(Inkscape) Color is used for the text in Inkscape, but the package 'color.sty' is not loaded}%
    \renewcommand\color[2][]{}%
  }%
  \providecommand\transparent[1]{%
    \errmessage{(Inkscape) Transparency is used (non-zero) for the text in Inkscape, but the package 'transparent.sty' is not loaded}%
    \renewcommand\transparent[1]{}%
  }%
  \providecommand\rotatebox[2]{#2}%
  \ifx\svgwidth\undefined%
    \setlength{\unitlength}{240bp}%
    \ifx\svgscale\undefined%
      \relax%
    \else%
      \setlength{\unitlength}{\unitlength * \real{\svgscale}}%
    \fi%
  \else%
    \setlength{\unitlength}{\svgwidth}%
  \fi%
  \global\let\svgwidth\undefined%
  \global\let\svgscale\undefined%
  \makeatother%
  \begin{picture}(1,0.6)%
    \put(0,0){\includegraphics[width=\unitlength,page=1]{circle_caustics1c.pdf}}%
    \put(0.3794648,0.51622929){\color[rgb]{0,0,0}\makebox(0,0)[lb]{\smash{\s(a)}}}%
  \end{picture}%
\endgroup%

%% file: circle_caustics2.pdf_tex
\begingroup%
  \makeatletter%
  \providecommand\color[2][]{%
    \errmessage{(Inkscape) Color is used for the text in Inkscape, but the package 'color.sty' is not loaded}%
    \renewcommand\color[2][]{}%
  }%
  \providecommand\transparent[1]{%
    \errmessage{(Inkscape) Transparency is used (non-zero) for the text in Inkscape, but the package 'transparent.sty' is not loaded}%
    \renewcommand\transparent[1]{}%
  }%
  \providecommand\rotatebox[2]{#2}%
  \ifx\svgwidth\undefined%
    \setlength{\unitlength}{342.8453125bp}%
    \ifx\svgscale\undefined%
      \relax%
    \else%
      \setlength{\unitlength}{\unitlength * \real{\svgscale}}%
    \fi%
  \else%
    \setlength{\unitlength}{\svgwidth}%
  \fi%
  \global\let\svgwidth\undefined%
  \global\let\svgscale\undefined%
  \makeatother%
  \begin{picture}(1,0.93336554)%
    \put(0,0){\includegraphics[width=\unitlength,page=1]{circle_caustics2.pdf}}%
    \put(-0.00136208,0.77528192){\color[rgb]{0,0,0}\makebox(0,0)[lb]{\smash{\s(b)}}}%
  \end{picture}%
\endgroup%

%% file: involute2.pdf_tex
\begingroup%
  \makeatletter%
  \providecommand\color[2][]{%
    \errmessage{(Inkscape) Color is used for the text in Inkscape, but the package 'color.sty' is not loaded}%
    \renewcommand\color[2][]{}%
  }%
  \providecommand\transparent[1]{%
    \errmessage{(Inkscape) Transparency is used (non-zero) for the text in Inkscape, but the package 'transparent.sty' is not loaded}%
    \renewcommand\transparent[1]{}%
  }%
  \providecommand\rotatebox[2]{#2}%
  \ifx\svgwidth\undefined%
    \setlength{\unitlength}{720bp}%
    \ifx\svgscale\undefined%
      \relax%
    \else%
      \setlength{\unitlength}{\unitlength * \real{\svgscale}}%
    \fi%
  \else%
    \setlength{\unitlength}{\svgwidth}%
  \fi%
  \global\let\svgwidth\undefined%
  \global\let\svgscale\undefined%
  \makeatother%
  \begin{picture}(1,0.43333333)%
    \put(0,0){\includegraphics[width=\unitlength,page=1]{involute2.pdf}}%
    \put(0.40504882,0.01840324){\color[rgb]{0,0,0}\makebox(0,0)[lb]{\smash{\s$\Gamma_2$}}}%
    \put(0.67686062,0.3156331){\color[rgb]{0,0,0}\makebox(0,0)[lb]{\smash{\s$\Delta_2$}}}%
    \put(0.32403999,0.22750613){\color[rgb]{0,0,0}\makebox(0,0)[lb]{\smash{\s$\gamma$}}}%
  \end{picture}%
\endgroup%

%% file: string.pdf_tex
\begingroup%
  \makeatletter%
  \providecommand\color[2][]{%
    \errmessage{(Inkscape) Color is used for the text in Inkscape, but the package 'color.sty' is not loaded}%
    \renewcommand\color[2][]{}%
  }%
  \providecommand\transparent[1]{%
    \errmessage{(Inkscape) Transparency is used (non-zero) for the text in Inkscape, but the package 'transparent.sty' is not loaded}%
    \renewcommand\transparent[1]{}%
  }%
  \providecommand\rotatebox[2]{#2}%
  \newcommand*\fsize{\dimexpr\f@size pt\relax}%
  \newcommand*\lineheight[1]{\fontsize{\fsize}{#1\fsize}\selectfont}%
  \ifx\svgwidth\undefined%
    \setlength{\unitlength}{482.90940857bp}%
    \ifx\svgscale\undefined%
      \relax%
    \else%
      \setlength{\unitlength}{\unitlength * \real{\svgscale}}%
    \fi%
  \else%
    \setlength{\unitlength}{\svgwidth}%
  \fi%
  \global\let\svgwidth\undefined%
  \global\let\svgscale\undefined%
  \makeatother%
  \begin{picture}(1,0.77415495)%
    \lineheight{1}%
    \setlength\tabcolsep{0pt}%
    \put(0,0){\includegraphics[width=\unitlength,page=1]{string.pdf}}%
    \put(0.92192295,0.07151241){\color[rgb]{0,0,0}\makebox(0,0)[lt]{\lineheight{1.25}\smash{\begin{tabular}[t]{l}\s$A$\end{tabular}}}}%
    \put(0.06067561,0.33859589){\color[rgb]{0,0,0}\makebox(0,0)[lt]{\lineheight{1.25}\smash{\begin{tabular}[t]{l}\s$\Delta_1$\end{tabular}}}}%
    \put(0.08390026,0.61729164){\color[rgb]{0,0,0}\makebox(0,0)[lt]{\lineheight{1.25}\smash{\begin{tabular}[t]{l}\s$Z$\end{tabular}}}}%
    \put(0.23679585,0.22247263){\color[rgb]{0,0,0}\makebox(0,0)[lt]{\lineheight{1.25}\smash{\begin{tabular}[t]{l}\s$\gamma$\end{tabular}}}}%
    \put(0.27356822,0.38891594){\color[rgb]{0,0,0}\makebox(0,0)[lt]{\lineheight{1.25}\smash{\begin{tabular}[t]{l}\s$X$\end{tabular}}}}%
    \put(0.40904535,0.05215853){\color[rgb]{0,0,0}\makebox(0,0)[lt]{\lineheight{1.25}\smash{\begin{tabular}[t]{l}\s$O$\end{tabular}}}}%
    \put(0.697418,0.03086927){\color[rgb]{0,0,0}\makebox(0,0)[lt]{\lineheight{1.25}\smash{\begin{tabular}[t]{l}\s$B$\end{tabular}}}}%
    \put(0.67419341,0.22440802){\color[rgb]{0,0,0}\makebox(0,0)[lt]{\lineheight{1.25}\smash{\begin{tabular}[t]{l}\s$\Gamma_1$\end{tabular}}}}%
    \put(0.03938635,0.04248157){\color[rgb]{0,0,0}\makebox(0,0)[lt]{\lineheight{1.25}\smash{\begin{tabular}[t]{l}\s(a)\end{tabular}}}}%
  \end{picture}%
\endgroup%

%% file: osculating_kepler_conic1.pdf_tex
\begingroup%
  \makeatletter%
  \providecommand\color[2][]{%
    \errmessage{(Inkscape) Color is used for the text in Inkscape, but the package 'color.sty' is not loaded}%
    \renewcommand\color[2][]{}%
  }%
  \providecommand\transparent[1]{%
    \errmessage{(Inkscape) Transparency is used (non-zero) for the text in Inkscape, but the package 'transparent.sty' is not loaded}%
    \renewcommand\transparent[1]{}%
  }%
  \providecommand\rotatebox[2]{#2}%
  \newcommand*\fsize{\dimexpr\f@size pt\relax}%
  \newcommand*\lineheight[1]{\fontsize{\fsize}{#1\fsize}\selectfont}%
  \ifx\svgwidth\undefined%
    \setlength{\unitlength}{284.93767548bp}%
    \ifx\svgscale\undefined%
      \relax%
    \else%
      \setlength{\unitlength}{\unitlength * \real{\svgscale}}%
    \fi%
  \else%
    \setlength{\unitlength}{\svgwidth}%
  \fi%
  \global\let\svgwidth\undefined%
  \global\let\svgscale\undefined%
  \makeatother%
  \begin{picture}(1,1.03952063)%
    \lineheight{1}%
    \setlength\tabcolsep{0pt}%
    \put(0,0){\includegraphics[width=\unitlength,page=1]{osculating_kepler_conic1.pdf}}%
    \put(0.24388128,0.21908658){\color[rgb]{0,0,0}\makebox(0,0)[lt]{\lineheight{1.25}\smash{\begin{tabular}[t]{l}\s$O$\end{tabular}}}}%
    \put(0.79189587,0.69761231){\color[rgb]{0,0,0}\makebox(0,0)[lt]{\lineheight{1.25}\smash{\begin{tabular}[t]{l}\s$\Gamma_1$\end{tabular}}}}%
    \put(0.39391405,0.35174725){\color[rgb]{0,0,0}\makebox(0,0)[lt]{\lineheight{1.25}\smash{\begin{tabular}[t]{l}\s$B$\end{tabular}}}}%
    \put(0.09226919,0.65970931){\color[rgb]{0,0,0}\makebox(0,0)[lt]{\lineheight{1.25}\smash{\begin{tabular}[t]{l}\s$X$\end{tabular}}}}%
    \put(0.77610297,0.10221895){\color[rgb]{0,0,0}\makebox(0,0)[lt]{\lineheight{1.25}\smash{\begin{tabular}[t]{l}\s$\gamma$\end{tabular}}}}%
    \put(-0.00090904,0.02799222){\color[rgb]{0,0,0}\makebox(0,0)[lt]{\lineheight{1.25}\smash{\begin{tabular}[t]{l}\s(b)\end{tabular}}}}%
  \end{picture}%
\endgroup%

%% file: lines.pdf_tex
\begingroup%
  \makeatletter%
  \providecommand\color[2][]{%
    \errmessage{(Inkscape) Color is used for the text in Inkscape, but the package 'color.sty' is not loaded}%
    \renewcommand\color[2][]{}%
  }%
  \providecommand\transparent[1]{%
    \errmessage{(Inkscape) Transparency is used (non-zero) for the text in Inkscape, but the package 'transparent.sty' is not loaded}%
    \renewcommand\transparent[1]{}%
  }%
  \providecommand\rotatebox[2]{#2}%
  \newcommand*\fsize{\dimexpr\f@size pt\relax}%
  \newcommand*\lineheight[1]{\fontsize{\fsize}{#1\fsize}\selectfont}%
  \ifx\svgwidth\undefined%
    \setlength{\unitlength}{525bp}%
    \ifx\svgscale\undefined%
      \relax%
    \else%
      \setlength{\unitlength}{\unitlength * \real{\svgscale}}%
    \fi%
  \else%
    \setlength{\unitlength}{\svgwidth}%
  \fi%
  \global\let\svgwidth\undefined%
  \global\let\svgscale\undefined%
  \makeatother%
  \begin{picture}(1,0.25714286)%
    \lineheight{1}%
    \setlength\tabcolsep{0pt}%
    \put(0,0){\includegraphics[width=\unitlength,page=1]{lines.pdf}}%
    \put(0.43467243,0.17593626){\color[rgb]{0,0,0}\makebox(0,0)[lt]{\lineheight{1.25}\smash{\begin{tabular}[t]{l}\s $p<0$\end{tabular}}}}%
    \put(0.44776401,0.10175066){\color[rgb]{0,0,0}\makebox(0,0)[lt]{\lineheight{1.25}\smash{\begin{tabular}[t]{l}\s $p>0$\end{tabular}}}}%
    \put(0.61249969,0.23593934){\color[rgb]{0,0,0}\makebox(0,0)[lt]{\lineheight{1.25}\smash{\begin{tabular}[t]{l}\s $\alpha$\end{tabular}}}}%
    \put(0.63868284,0.11702417){\color[rgb]{0,0,0}\makebox(0,0)[lt]{\lineheight{1.25}\smash{\begin{tabular}[t]{l}\s $\alpha$\end{tabular}}}}%
    \put(0.38339709,0.12466093){\color[rgb]{0,0,0}\makebox(0,0)[lt]{\lineheight{1.25}\smash{\begin{tabular}[t]{l}\s $O$\end{tabular}}}}%
  \end{picture}%
\endgroup%

%% file: phase_cylinder6.pdf_tex
\begingroup%
  \makeatletter%
  \providecommand\color[2][]{%
    \errmessage{(Inkscape) Color is used for the text in Inkscape, but the package 'color.sty' is not loaded}%
    \renewcommand\color[2][]{}%
  }%
  \providecommand\transparent[1]{%
    \errmessage{(Inkscape) Transparency is used (non-zero) for the text in Inkscape, but the package 'transparent.sty' is not loaded}%
    \renewcommand\transparent[1]{}%
  }%
  \providecommand\rotatebox[2]{#2}%
  \ifx\svgwidth\undefined%
    \setlength{\unitlength}{376.41169434bp}%
    \ifx\svgscale\undefined%
      \relax%
    \else%
      \setlength{\unitlength}{\unitlength * \real{\svgscale}}%
    \fi%
  \else%
    \setlength{\unitlength}{\svgwidth}%
  \fi%
  \global\let\svgwidth\undefined%
  \global\let\svgscale\undefined%
  \makeatother%
  \begin{picture}(1,1.14767954)%
    \put(0,0){\includegraphics[width=\unitlength,page=1]{phase_cylinder6.pdf}}%
    \put(0.03534857,0.57004396){\color[rgb]{0,0,0}\makebox(0,0)[lb]{\smash{\s $C_0$}}}%
    \put(0.04718517,0.30402309){\color[rgb]{0,0,0}\makebox(0,0)[lb]{\smash{\s $M$}}}%
    \put(0.45227795,0.43800577){\color[rgb]{0,0,0}\makebox(0,0)[lb]{\smash{\s $\alpha$}}}%
    \put(0.93122473,0.5523195){\color[rgb]{0,0,0}\makebox(0,0)[lb]{\smash{\s $p=0$}}}%
    \put(0,0){\includegraphics[width=\unitlength,page=2]{phase_cylinder6.pdf}}%
    \put(0.07724701,0.77100528){\color[rgb]{0,0,0}\makebox(0,0)[lb]{\smash{\s $p$}}}%
    \put(0.928399,0.21266021){\color[rgb]{0,0,0}\makebox(0,0)[lb]{\smash{\s $\L$}}}%
  \end{picture}%
\endgroup%

%% file: T.pdf_tex
\begingroup%
  \makeatletter%
  \providecommand\color[2][]{%
    \errmessage{(Inkscape) Color is used for the text in Inkscape, but the package 'color.sty' is not loaded}%
    \renewcommand\color[2][]{}%
  }%
  \providecommand\transparent[1]{%
    \errmessage{(Inkscape) Transparency is used (non-zero) for the text in Inkscape, but the package 'transparent.sty' is not loaded}%
    \renewcommand\transparent[1]{}%
  }%
  \providecommand\rotatebox[2]{#2}%
  \ifx\svgwidth\undefined%
    \setlength{\unitlength}{315.23702787bp}%
    \ifx\svgscale\undefined%
      \relax%
    \else%
      \setlength{\unitlength}{\unitlength * \real{\svgscale}}%
    \fi%
  \else%
    \setlength{\unitlength}{\svgwidth}%
  \fi%
  \global\let\svgwidth\undefined%
  \global\let\svgscale\undefined%
  \makeatother%
  \begin{picture}(1,0.89921089)%
    \put(0,0){\includegraphics[width=\unitlength,page=1]{T.pdf}}%
    \put(0.43582995,0.58734555){\color[rgb]{0,0,0}\makebox(0,0)[lb]{\smash{\s$r$}}}%
    \put(0.71921921,0.44919321){\color[rgb]{0,0,0}\makebox(0,0)[lb]{\smash{\s$T(r)$}}}%
    \put(0.00011897,0.2083125){\color[rgb]{0,0,0}\makebox(0,0)[lb]{\smash{\s$\gamma$}}}%
    \put(-0.18774538,0.83599828){\color[rgb]{0,0,0}\makebox(0,0)[lt]{\begin{minipage}{1.39923449\unitlength}\raggedright \end{minipage}}}%
  \end{picture}%
\endgroup%

%% file: tphi.pdf_tex
\begingroup%
  \makeatletter%
  \providecommand\color[2][]{%
    \errmessage{(Inkscape) Color is used for the text in Inkscape, but the package 'color.sty' is not loaded}%
    \renewcommand\color[2][]{}%
  }%
  \providecommand\transparent[1]{%
    \errmessage{(Inkscape) Transparency is used (non-zero) for the text in Inkscape, but the package 'transparent.sty' is not loaded}%
    \renewcommand\transparent[1]{}%
  }%
  \providecommand\rotatebox[2]{#2}%
  \ifx\svgwidth\undefined%
    \setlength{\unitlength}{250bp}%
    \ifx\svgscale\undefined%
      \relax%
    \else%
      \setlength{\unitlength}{\unitlength * \real{\svgscale}}%
    \fi%
  \else%
    \setlength{\unitlength}{\svgwidth}%
  \fi%
  \global\let\svgwidth\undefined%
  \global\let\svgscale\undefined%
  \makeatother%
  \begin{picture}(1,0.48)%
    \put(0,0){\includegraphics[width=\unitlength,page=1]{tphi.pdf}}%
    \put(0.76760076,0.18250316){\color[rgb]{0,0,0}\makebox(0,0)[lb]{\smash{'}}}%
    \put(0,0){\includegraphics[width=\unitlength,page=2]{tphi.pdf}}%
    \put(0.82470557,0.11002918){\color[rgb]{0,0,0}\makebox(0,0)[lb]{\smash{$\s\gamma(t)$}}}%
    \put(0.92122264,0.20496398){\color[rgb]{0,0,0}\makebox(0,0)[lb]{\smash{$\s\gamma'(t)$}}}%
    \put(0.31996856,0.26508946){\color[rgb]{0,0,0}\makebox(0,0)[lb]{\smash{\s $p$}}}%
    \put(0.2772237,0.13810803){\color[rgb]{0,0,0}\makebox(0,0)[lb]{\smash{\s$O$}}}%
    \put(0.59909548,0.27969493){\color[rgb]{0,0,0}\makebox(0,0)[lb]{\smash{\s$\alpha$}}}%
    \put(0.75164161,0.19206213){\color[rgb]{0,0,0}\makebox(0,0)[lb]{\smash{$\s\varphi$}}}%
    \put(0.02398838,0.01290248){\color[rgb]{0,0,0}\makebox(0,0)[lb]{\smash{\s(a)}}}%
  \end{picture}%
\endgroup%

%% file: chord.pdf_tex
\begingroup%
  \makeatletter%
  \providecommand\color[2][]{%
    \errmessage{(Inkscape) Color is used for the text in Inkscape, but the package 'color.sty' is not loaded}%
    \renewcommand\color[2][]{}%
  }%
  \providecommand\transparent[1]{%
    \errmessage{(Inkscape) Transparency is used (non-zero) for the text in Inkscape, but the package 'transparent.sty' is not loaded}%
    \renewcommand\transparent[1]{}%
  }%
  \providecommand\rotatebox[2]{#2}%
  \ifx\svgwidth\undefined%
    \setlength{\unitlength}{240bp}%
    \ifx\svgscale\undefined%
      \relax%
    \else%
      \setlength{\unitlength}{\unitlength * \real{\svgscale}}%
    \fi%
  \else%
    \setlength{\unitlength}{\svgwidth}%
  \fi%
  \global\let\svgwidth\undefined%
  \global\let\svgscale\undefined%
  \makeatother%
  \begin{picture}(1,0.5)%
    \put(0,0){\includegraphics[width=\unitlength,page=1]{chord.pdf}}%
    \put(0.8049298,0.08476513){\color[rgb]{0,0,0}\makebox(0,0)[lb]{\smash{\s$\gamma(t)$}}}%
    \put(0.9667628,0.24705557){\color[rgb]{0,0,0}\makebox(0,0)[lb]{\smash{\s$\gamma'(t)$}}}%
    \put(0.63337543,0.25631949){\color[rgb]{0,0,0}\makebox(0,0)[lb]{\smash{\s$\alpha$}}}%
    \put(0.74726358,0.19150518){\color[rgb]{0,0,0}\makebox(0,0)[lb]{\smash{\s$\varphi$}}}%
    \put(0.45061746,0.39023209){\color[rgb]{0,0,0}\makebox(0,0)[lb]{\smash{\s$\varphi_1$}}}%
    \put(0.23546453,0.36997914){\color[rgb]{0,0,0}\makebox(0,0)[lb]{\smash{\s$\varphi_1$}}}%
    \put(0.34506875,0.47028103){\color[rgb]{0,0,0}\makebox(0,0)[lb]{\smash{\s$\gamma(t_1)$}}}%
    \put(0.65047704,-0.0333578){\color[rgb]{0,0,0}\makebox(0,0)[lb]{\smash{}}}%
    \put(0.02340431,0.02519764){\color[rgb]{0,0,0}\makebox(0,0)[lb]{\smash{\s(b)}}}%
    \put(0.49625018,0.22558451){\color[rgb]{0,0,0}\makebox(0,0)[lb]{\smash{\s$L$}}}%
  \end{picture}%
\endgroup%

%% file: 3beams.pdf_tex
\begingroup%
  \makeatletter%
  \providecommand\color[2][]{%
    \errmessage{(Inkscape) Color is used for the text in Inkscape, but the package 'color.sty' is not loaded}%
    \renewcommand\color[2][]{}%
  }%
  \providecommand\transparent[1]{%
    \errmessage{(Inkscape) Transparency is used (non-zero) for the text in Inkscape, but the package 'transparent.sty' is not loaded}%
    \renewcommand\transparent[1]{}%
  }%
  \providecommand\rotatebox[2]{#2}%
  \ifx\svgwidth\undefined%
    \setlength{\unitlength}{1040bp}%
    \ifx\svgscale\undefined%
      \relax%
    \else%
      \setlength{\unitlength}{\unitlength * \real{\svgscale}}%
    \fi%
  \else%
    \setlength{\unitlength}{\svgwidth}%
  \fi%
  \global\let\svgwidth\undefined%
  \global\let\svgscale\undefined%
  \makeatother%
  \begin{picture}(1,0.38533738)%
    \put(0,0){\includegraphics[width=\unitlength,page=1]{3beams.pdf}}%
    \put(0.21105142,0.3633998){\color[rgb]{0,0,0}\makebox(0,0)[lb]{\smash{\s$p$}}}%
    \put(0.83162643,0.18868836){\color[rgb]{0,0,0}\makebox(0,0)[lb]{\smash{\s$\alpha$}}}%
    \put(0.21524451,0.03633998){\color[rgb]{0,0,0}\makebox(0,0)[lb]{\smash{\s$C_1$}}}%
    \put(0.46123822,0.03633998){\color[rgb]{0,0,0}\makebox(0,0)[lb]{\smash{\s$C_2$}}}%
    \put(0.69884574,0.03633998){\color[rgb]{0,0,0}\makebox(0,0)[lb]{\smash{\s$C_3$}}}%
    \put(0.53671353,0.04332844){\color[rgb]{0,0,0}\makebox(0,0)[lb]{\smash{}}}%
  \end{picture}%
\endgroup%

%% file: contact.pdf_tex
\begingroup%
  \makeatletter%
  \providecommand\color[2][]{%
    \errmessage{(Inkscape) Color is used for the text in Inkscape, but the package 'color.sty' is not loaded}%
    \renewcommand\color[2][]{}%
  }%
  \providecommand\transparent[1]{%
    \errmessage{(Inkscape) Transparency is used (non-zero) for the text in Inkscape, but the package 'transparent.sty' is not loaded}%
    \renewcommand\transparent[1]{}%
  }%
  \providecommand\rotatebox[2]{#2}%
  \ifx\svgwidth\undefined%
    \setlength{\unitlength}{240bp}%
    \ifx\svgscale\undefined%
      \relax%
    \else%
      \setlength{\unitlength}{\unitlength * \real{\svgscale}}%
    \fi%
  \else%
    \setlength{\unitlength}{\svgwidth}%
  \fi%
  \global\let\svgwidth\undefined%
  \global\let\svgscale\undefined%
  \makeatother%
  \begin{picture}(1,0.33333333)%
    \put(0,0){\includegraphics[width=\unitlength,page=1]{contact.pdf}}%
    \put(0.59820918,0.24049924){\color[rgb]{0,0,0}\makebox(0,0)[lb]{\smash{$A$}}}%
    \put(0.42673525,0.11050657){\color[rgb]{0,0,0}\makebox(0,0)[lb]{\smash{$p$}}}%
    \put(0.64062022,0.15669975){\color[rgb]{0,0,0}\makebox(0,0)[lb]{\smash{$\ell$}}}%
    \put(0.56559607,0.03783201){\color[rgb]{0,0,0}\makebox(0,0)[lb]{\smash{$z$}}}%
    \put(0.53556881,0.28463681){\color[rgb]{0,0,0}\makebox(0,0)[lb]{\smash{$\theta$}}}%
    \put(0.45377967,0.01220742){\color[rgb]{0,0,0}\makebox(0,0)[lb]{\smash{$O$}}}%
    \put(0.33561563,0.2062753){\color[rgb]{0,0,0}\makebox(0,0)[lb]{\smash{$\alpha$}}}%
  \end{picture}%
\endgroup%

%% file: diagram1.pdf_tex
\begingroup%
  \makeatletter%
  \providecommand\color[2][]{%
    \errmessage{(Inkscape) Color is used for the text in Inkscape, but the package 'color.sty' is not loaded}%
    \renewcommand\color[2][]{}%
  }%
  \providecommand\transparent[1]{%
    \errmessage{(Inkscape) Transparency is used (non-zero) for the text in Inkscape, but the package 'transparent.sty' is not loaded}%
    \renewcommand\transparent[1]{}%
  }%
  \providecommand\rotatebox[2]{#2}%
  \ifx\svgwidth\undefined%
    \setlength{\unitlength}{794.27752637bp}%
    \ifx\svgscale\undefined%
      \relax%
    \else%
      \setlength{\unitlength}{\unitlength * \real{\svgscale}}%
    \fi%
  \else%
    \setlength{\unitlength}{\svgwidth}%
  \fi%
  \global\let\svgwidth\undefined%
  \global\let\svgscale\undefined%
  \makeatother%
  \begin{picture}(1,0.55989159)%
    \put(0.48215115,0.43773225){\color[rgb]{0,0,0}\makebox(0,0)[lb]{\smash{$C$}}}%
    \put(0.23413377,0.11088534){\color[rgb]{0,0,0}\makebox(0,0)[lb]{\smash{$\Delta$}}}%
    \put(0.72776836,0.11206962){\color[rgb]{0,0,0}\makebox(0,0)[lb]{\smash{$\Gamma$}}}%
    \put(0,0){\includegraphics[width=\unitlength,page=1]{diagram1.pdf}}%
    \put(0.52337087,0.53612785){\color[rgb]{0,0,0}\makebox(0,0)[lb]{\smash{{\scriptsize inflection pts}}}}%
    \put(0,0){\includegraphics[width=\unitlength,page=2]{diagram1.pdf}}%
    \put(0.79538918,0.00660977){\color[rgb]{0,0,0}\makebox(0,0)[lb]{\smash{{\scriptsize cusps}}}}%
    \put(0.10265496,0.00698779){\color[rgb]{0,0,0}\makebox(0,0)[lb]{\smash{{\scriptsize vertices}}}}%
    \put(0,0){\includegraphics[width=\unitlength,page=3]{diagram1.pdf}}%
    \put(0.43401859,0.12710595){\color[rgb]{0,0,0}\makebox(0,0)[lb]{\smash{{\scriptsize projective}}}}%
    \put(0.44391881,0.05799907){\color[rgb]{0,0,0}\makebox(0,0)[lb]{\smash{{\scriptsize duality}}}}%
    \put(0.49603015,-0.05166663){\color[rgb]{0,0,0}\makebox(0,0)[lt]{\begin{minipage}{0.14360436\unitlength}\raggedright \end{minipage}}}%
    \put(0.70079945,-0.08623802){\color[rgb]{0,0,0}\makebox(0,0)[lt]{\begin{minipage}{0.03989007\unitlength}\raggedright \end{minipage}}}%
    \put(0.27072389,0.23193666){\color[rgb]{0,0,0}\rotatebox{55.14744267}{\makebox(0,0)[lb]{\smash{{\scriptsize normals}}}}}%
    \put(0.36402315,0.17664409){\color[rgb]{0,0,0}\rotatebox{56.38721752}{\makebox(0,0)[lb]{\smash{{\scriptsize wave front}}}}}%
    \put(0.55231812,0.35248808){\color[rgb]{0,0,0}\rotatebox{-56.13913692}{\makebox(0,0)[lb]{\smash{{\scriptsize involute}}}}}%
    \put(0.64938405,0.39503754){\color[rgb]{0,0,0}\rotatebox{-56.13913692}{\makebox(0,0)[lb]{\smash{{\scriptsize evvolute}}}}}%
    \put(-0.00121486,0.43489169){\color[rgb]{0,0,0}\makebox(0,0)[lb]{\smash{$\widetilde C$}}}%
    \put(0,0){\includegraphics[width=\unitlength,page=4]{diagram1.pdf}}%
    \put(0.18188411,0.46614244){\color[rgb]{0,0,0}\makebox(0,0)[lb]{\smash{\s$\pi_2$}}}%
    \put(0.04149042,0.29560716){\color[rgb]{0,0,0}\makebox(0,0)[lb]{\smash{\s$\pi_1$}}}%
  \end{picture}%
\endgroup%